\documentclass[11pt]{article}
\usepackage{fullpage}
\usepackage{amsmath,graphicx,amsthm,amsfonts,amssymb,amstext}
\usepackage{ragged2e, enumerate}
\usepackage[T1]{fontenc}

\usepackage{graphicx,lscape}
\usepackage{pgf,tikz}
\usetikzlibrary{decorations.pathreplacing, calligraphy}
\usepackage{float}
\usepackage{subcaption}
\usetikzlibrary{graphs,graphs.standard}
\usepackage[symbols,nogroupskip,sort=none]{glossaries-extra}
\usepackage{subfiles}
\usepackage{framed}

\newtheorem{theorem}{Theorem}[section]

\newtheorem{lemma}{Lemma}[section]

\newtheorem{proposition}{Proposition}[section]

\newtheorem{claim}{Claim}[section]

\newtheorem{corollary}{Corollary}[section]

\newtheorem{conjecture}{Conjecture}[section]

\theoremstyle{remark}

\usepackage{hyperref}
\hypersetup{
	colorlinks=true,
	linkcolor=blue,
}

\title{On the second largest adjacency eigenvalue of trees with given diameter}

\author{Hitesh Kumar \and Bojan Mohar \and Shivaramakrishna Pragada \and Hanmeng Zhan}

\date{}

\begin{document}
\maketitle

\begin{abstract}
For a graph $G$, let $\lambda_2(G)$ denote the second largest eigenvalue of the adjacency matrix of $G$. We determine the extremal trees with maximum/minimum adjacency eigenvalue $\lambda_2$ in the class $\mathcal{T}(n,d)$ of $n$-vertex trees with diameter $d$. This contributes to the literature on $\lambda_2$-extremization over different graph families. We also revisit the notion of the spectral center of a tree and the proof of $\lambda_2$ maximization over trees. 
\end{abstract}

\noindent
\textbf{Keywords:} adjacency matrix, second largest adjacency eigenvalue, diameter, spectral center, caterpillar, Smith graphs

\noindent
\textbf{MSC:} 05C50

\section{Introduction}
\label{section:introduction}
\subsection{Notation}\label{subsection:notation}
All graphs considered in this paper have at least four vertices unless stated otherwise. For a simple graph $G=(V(G), E(G))$ on $n=|G|$ vertices, let $A(G)$ denote its \emph{adjacency matrix}. The \emph{spectrum} of $G$ is the multiset of eigenvalues of $A(G)$, which we enumerate in non-increasing order as 
\[\lambda_1(G)\geq \lambda_2(G)\geq \cdots \geq \lambda_n(G).\] 
Let $\mathcal{G}(n)$ and $\mathcal{T}(n)$ denote the set of connected simple graphs and the set of trees, respectively, on $n$ vertices. For $2 \le d\le n-1$, let $\mathcal{T}(n,d)$ denote the set of trees with diameter $d$ on $n$ vertices. For a given family $\mathcal F$ (which will usually be one of $\mathcal{G}(n)$, $\mathcal{T}(n)$, or $\mathcal{T}(n,d)$), we say that a graph $G\in \mathcal F$ is a \emph{$\lambda_i$-maximizer} (\emph{$\lambda_i$-minimizer}) if for every $H\in \mathcal F$, we have $\lambda_i(H) \le \lambda_i(G)$ ($\lambda_i(H) \ge \lambda_i(G)$).

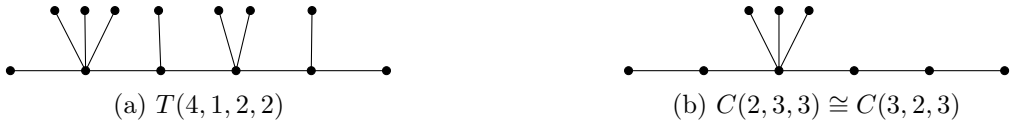
\begin{figure}[H]
	\begin{subfigure}{0.49\textwidth}
		\centering
		\begin{tikzpicture}[scale=1]
			\draw  (-3,0)-- (-2,0);
			\draw  (-2,0)-- (-1,0);
			\draw  (-1,0)-- (0,0);
			\draw  (0,0)-- (1,0);
			\draw  (1,0)-- (2,0);
			\draw  (-2.41,0.8)-- (-2,0);
			\draw  (-2.01,0.8)-- (-2,0);
			\draw  (-1.61,0.8)-- (-2,0);
			\draw  (-1.03,0.8)-- (-1,0);
			\draw  (-0.23,0.8)-- (0,0);
			\draw  (0.19,0.8)-- (0,0);
			\draw  (1.01,0.8)-- (1,0);
			\draw [fill=black] (-3,0) circle (1.5pt);
			\draw [fill=black] (-2,0) circle (1.5pt);
			\draw [fill=black] (-1,0) circle (1.5pt);
			\draw [fill=black] (0,0) circle (1.5pt);
			\draw [fill=black] (1,0) circle (1.5pt);
			\draw [fill=black] (2,0) circle (1.5pt);
			\draw [fill=black] (-2.41,0.8) circle (1.5pt);
			\draw [fill=black] (-2.01,0.8) circle (1.5pt);
			\draw [fill=black] (-1.61,0.8) circle (1.5pt);
			\draw [fill=black] (-1.03,0.8) circle (1.5pt);
			\draw [fill=black] (-0.23,0.8) circle (1.5pt);
			\draw [fill=black] (0.19,0.8) circle (1.5pt);
			\draw [fill=black] (1.01,0.8) circle (1.5pt);
		\end{tikzpicture}
		\caption{$T(4,1,2,2)$}    
	\end{subfigure}
	\begin{subfigure}{0.49\textwidth}
		\centering
		\begin{tikzpicture}[scale=1]
			\draw  (-2,0)-- (-1,0);
			\draw  (0,0)-- (-1,0);
			\draw  (0,0)-- (1,0);
			\draw  (1,0)-- (2,0);
			\draw  (2,0)-- (3,0);
			\draw  (-0.4,0.8)-- (0,0);
			\draw  (0,0.8)-- (0,0);
			\draw  (0.4,0.8)-- (0,0);
			\draw [fill=black] (-2,0) circle (1.5pt);
			\draw [fill=black] (-1,0) circle (1.5pt);
			\draw [fill=black] (0,0) circle (1.5pt);
			\draw [fill=black] (1,0) circle (1.5pt);
			\draw [fill=black] (2,0) circle (1.5pt);
			\draw [fill=black] (3,0) circle (1.5pt);
			\draw [fill=black] (-0.4,0.8) circle (1.5pt);
			\draw [fill=black] (0,0.8) circle (1.5pt);
			\draw [fill=black] (0.4,0.8) circle (1.5pt);
		\end{tikzpicture}
		\caption{$C(2,3,3)\cong C(3,2,3)$}    
	\end{subfigure}
	\caption{Examples of caterpillars}
	\label{fig:caterpillar}
\end{figure}

A \emph{caterpillar} is a tree in which removing all pendant vertices gives a path. We denote by $T(m_1,\ldots,m_{d-1})$ the caterpillar obtained by attaching $m_i$ leaves to the $i$-th vertex $v_i$ of a path $v_1 v_2 \ldots v_{d-1}$. We assume that $m_1\ge 1$ and $m_{d-1}\ge 1$ so that the obtained caterpillar has diameter equal to $d$. Other $m_i$ can be $0$. Note here that $n=d-1+\sum_{i=1}^{d-1} m_i$ and $T(m_1,\ldots,m_{d-1})\in \mathcal{T}(n,d)$. For $\ell,r\ge 0$, we use the shorthand $C(\ell,r,k)$ to denote the caterpillar obtained by taking the path $v_0 \ldots v_{\ell+r}$ of length $\ell+r$ and attaching $k$ leaves at the $\ell$-th vertex. Clearly $C(\ell,r,k) \in \mathcal{T}(\ell+r+k+1,\ell+r)$.

Let $(G_1,a)$ and $(G_2,b)$ be two (disjoint) rooted graphs with roots $a$ and $b$, respectively. Then $(G_1, a)\circ (G_2,b)$ denotes the graph obtained from their union by connecting the roots $a$ and $b$ by an edge. If $v=K_1$ then $(G_1,a)\circ v \circ (G_2,b)$ or $(G_1,a)\circ K_1 \circ (G_2,b)$ denotes the graph obtained by joining $a$ and $b$ to $v$. 

\begin{figure}[H]
	\begin{subfigure}{0.49\textwidth}
		\centering
		\begin{tikzpicture}[scale=0.7]
			\draw [fill=gray!10] (-2,0) circle (1cm);
			\draw [fill=gray!10] (2,0) circle (1cm);
			\draw  (-1,0)-- (0,0);
			\draw  (0,0)-- (1,0);
			
			\draw (-2,0) node {$G_1$};
			\draw (2,0) node {$G_2$};
			\draw (-0.85,-0.3) node {$a$};
			\draw (0.85,-0.3) node {$b$};
			\draw [fill=black] (-1,0) circle (1.5pt);
			\draw [fill=black] (1,0) circle (1.5pt);
		\end{tikzpicture}
		\caption{$(G_1,a)\circ (G_2,b)$}
	\end{subfigure}
	\begin{subfigure}{0.49\textwidth}
		\centering
		\begin{tikzpicture}[scale=0.7]
			\draw [fill=gray!10] (-2,0) circle (1cm);
			\draw [fill=gray!10] (2,0) circle (1cm);
			\draw  (-1,0)-- (0,0);
			\draw  (0,0)-- (1,0);
			
			\draw (-2,0) node {$G_1$};
			\draw (2,0) node {$G_2$};
			\draw (-0.85,-0.3) node {$a$};
			\draw (0.85,-0.3) node {$b$};
			\draw (0,-0.3) node {$v$};
			\draw [fill=black] (-1,0) circle (1.5pt);
			\draw [fill=black] (0,0) circle (1.5pt);
			\draw [fill=black] (1,0) circle (1.5pt);
		\end{tikzpicture}  
		\caption{$(G_1,a)\circ v \circ (G_2,b)$}
	\end{subfigure}
	\caption{}
\end{figure}

\subsection{Motivation}

A problem of interest in spectral graph theory is extremizing a given graph eigenvalue over a given graph family and determining the extremal graphs. Many results are known for extremization of the first eigenvalue $\lambda_1$ (or \emph{spectral radius}), see \cite{Cvetkovic_Rowlinson_1990} and \cite[Chapter 2]{Stanic_book}. The second eigenvalue $\lambda_2$ has been well-studied (especially for regular graphs) due to its connection with graph connectivity, see \cite{Cvetkovic_Simic_1995} and \cite[Chapter 4]{Stanic_book}. In recent years, the extremization of $\lambda_2$ has also received attention. 

One of the earliest results in this direction concerns the maximization of $\lambda_2$ in the class of $n$-vertex trees $\mathcal{T}(n)$, as described in \autoref{thm:lambda_two_max_trees}. This result was first proved by Neumaier \cite{neumaier_second_1982}, but it had a mistake, and the result was corrected by Hofmeister \cite{hofmeister_1997}.

\begin{theorem}[\cite{neumaier_second_1982, hofmeister_1997}]
	\label{thm:lambda_two_max_trees}
	Let $T$ be a $\lambda_2$-maximizer in $\mathcal{T}(n)$.  
	\begin{enumerate}[$(i)$]
		\item If $n$ is odd then $T$ is a caterpillar $T(\frac{n-3}{2},0,\frac{n-3}{2})$, $T(\frac{n-3}{2},0,0,\frac{n-5}{2})$, or $T(\frac{n-5}{2},0,0,0,\frac{n-5}{2})$.
		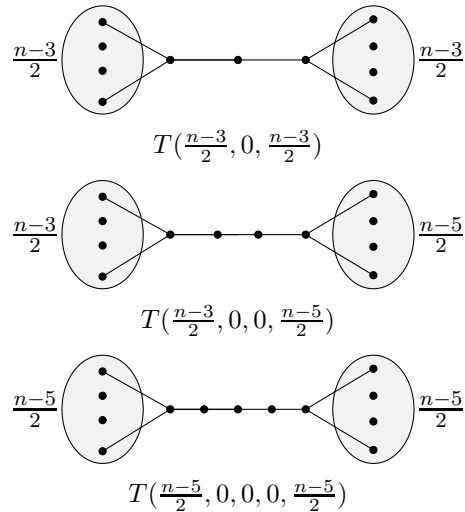
\begin{figure}[H]
			\begin{subfigure}{1.0\textwidth}
				\centering
				\begin{tikzpicture}[scale=0.9]
					\draw [rotate around={90:(2,0)}, fill=gray!10] (2,0) ellipse (0.8cm and 0.6cm);
					\draw [rotate around={90:(-2,0)}, fill=gray!10] (-2,0) ellipse (0.8cm and 0.6cm);
					\draw  (0,0)-- (-1,0);
					\draw  (-1,0)-- (-2,-0.62);
					\draw  (-1,0)-- (-2,0.56);
					\draw  (-1,0)-- (1,0);
					\draw  (1,0)-- (2,0.6);
					\draw  (1,0)-- (2,-0.6);
					
					\draw [fill=black] (0,0) circle (1.5pt);
					\draw [fill=black] (-1,0) circle (1.5pt);
					\draw [fill=black] (-2,-0.62) circle (1.5pt);
					\draw [fill=black] (-2,0.56) circle (1.5pt);
					\draw [fill=black] (1,0) circle (1.5pt);
					\draw [fill=black] (2,0.6) circle (1.5pt);
					\draw [fill=black] (2,-0.6) circle (1.5pt);
					\draw (3,0) node {$\frac{n-3}{2}$};
					\draw (-3,0) node {$\frac{n-3}{2}$};
					\draw [fill=black] (-2,0.2) circle (1.5pt);
					\draw [fill=black] (-2,-0.16) circle (1.5pt);
					\draw [fill=black] (2,0.2) circle (1.5pt);
					\draw [fill=black] (2,-0.18) circle (1.5pt);
				\end{tikzpicture}
				\caption*{$T(\frac{n-3}{2},0,\frac{n-3}{2})$}
			\end{subfigure}\\[0.2cm]
			\begin{subfigure}{1.0\textwidth}
				\centering
				\begin{tikzpicture}[scale=0.9]
					\draw [rotate around={90:(2,0)}, fill=gray!10] (2,0) ellipse (0.8cm and 0.6cm);
					\draw [rotate around={90:(-2,0)}, fill=gray!10] (-2,0) ellipse (0.8cm and 0.6cm);
					\draw  (0,0)-- (-1,0);
					\draw  (-1,0)-- (-2,-0.62);
					\draw  (-1,0)-- (-2,0.56);
					\draw  (-1,0)-- (1,0);
					\draw  (1,0)-- (2,0.6);
					\draw  (1,0)-- (2,-0.6);
					
					\draw [fill=black] (0.3,0) circle (1.5pt);
					\draw [fill=black] (-0.3,0) circle (1.5pt);
					\draw [fill=black] (-1,0) circle (1.5pt);
					\draw [fill=black] (-2,-0.62) circle (1.5pt);
					\draw [fill=black] (-2,0.56) circle (1.5pt);
					\draw [fill=black] (1,0) circle (1.5pt);
					\draw [fill=black] (2,0.6) circle (1.5pt);
					\draw [fill=black] (2,-0.6) circle (1.5pt);
					\draw (3,0) node {$\frac{n-5}{2}$};
					\draw (-3,0) node {$\frac{n-3}{2}$};
					\draw [fill=black] (-2,0.2) circle (1.5pt);
					\draw [fill=black] (-2,-0.16) circle (1.5pt);
					\draw [fill=black] (2,0.2) circle (1.5pt);
					\draw [fill=black] (2,-0.18) circle (1.5pt);
				\end{tikzpicture}
				\caption*{$T(\frac{n-3}{2},0,0,\frac{n-5}{2})$}
			\end{subfigure}\\[0.2cm]
			\begin{subfigure}{1.0\textwidth}
				\centering
				\begin{tikzpicture}[scale=0.9]
					\draw [rotate around={90:(2,0)}, fill=gray!10] (2,0) ellipse (0.8cm and 0.6cm);
					\draw [rotate around={90:(-2,0)}, fill=gray!10] (-2,0) ellipse (0.8cm and 0.6cm);
					\draw  (0,0)-- (-1,0);
					\draw  (-1,0)-- (-2,-0.62);
					\draw  (-1,0)-- (-2,0.56);
					\draw  (-1,0)-- (1,0);
					\draw  (1,0)-- (2,0.6);
					\draw  (1,0)-- (2,-0.6);
					
					\draw [fill=black] (0,0) circle (1.5pt);
					\draw [fill=black] (0.5,0) circle (1.5pt);
					\draw [fill=black] (-0.5,0) circle (1.5pt);
					\draw [fill=black] (-1,0) circle (1.5pt);
					\draw [fill=black] (-2,-0.62) circle (1.5pt);
					\draw [fill=black] (-2,0.56) circle (1.5pt);
					\draw [fill=black] (1,0) circle (1.5pt);
					\draw [fill=black] (2,0.6) circle (1.5pt);
					\draw [fill=black] (2,-0.6) circle (1.5pt);
					\draw (3,0) node {$\frac{n-5}{2}$};
					\draw (-3,0) node {$\frac{n-5}{2}$};
					\draw [fill=black] (-2,0.2) circle (1.5pt);
					\draw [fill=black] (-2,-0.16) circle (1.5pt);
					\draw [fill=black] (2,0.2) circle (1.5pt);
					\draw [fill=black] (2,-0.18) circle (1.5pt);
				\end{tikzpicture}
				\caption*{$T(\frac{n-5}{2},0,0,0,\frac{n-5}{2})$}
			\end{subfigure}
			\caption{$\lambda_2$-maximizers in $\mathcal{T}(n)$ when $n$ is odd}
		\end{figure}
		
		\item If $n$ is even then $T$ is the caterpillar $T(\frac{n-4}{2},0,0,\frac{n-4}{2})$.
		\begin{figure}[H]
			\begin{subfigure}{1.0\textwidth}
				\centering
				\begin{tikzpicture}[scale=0.9]
					\draw [rotate around={90:(2,0)}, fill=gray!10] (2,0) ellipse (0.8cm and 0.6cm);
					\draw [rotate around={90:(-2,0)}, fill=gray!10] (-2,0) ellipse (0.8cm and 0.6cm);
					\draw  (0,0)-- (-1,0);
					\draw  (-1,0)-- (-2,-0.62);
					\draw  (-1,0)-- (-2,0.56);
					\draw  (-1,0)-- (1,0);
					\draw  (1,0)-- (2,0.6);
					\draw  (1,0)-- (2,-0.6);
					
					\draw [fill=black] (0.3,0) circle (1.5pt);
					\draw [fill=black] (-0.3,0) circle (1.5pt);
					\draw [fill=black] (-1,0) circle (1.5pt);
					\draw [fill=black] (-2,-0.62) circle (1.5pt);
					\draw [fill=black] (-2,0.56) circle (1.5pt);
					\draw [fill=black] (1,0) circle (1.5pt);
					\draw [fill=black] (2,0.6) circle (1.5pt);
					\draw [fill=black] (2,-0.6) circle (1.5pt);
					\draw (3,0) node {$\frac{n-4}{2}$};
					\draw (-3,0) node {$\frac{n-4}{2}$};
					\draw [fill=black] (-2,0.2) circle (1.5pt);
					\draw [fill=black] (-2,-0.16) circle (1.5pt);
					\draw [fill=black] (2,0.2) circle (1.5pt);
					\draw [fill=black] (2,-0.18) circle (1.5pt);
				\end{tikzpicture}
				\caption*{$T(\frac{n-4}{2},0,0,\frac{n-4}{2})$}
			\end{subfigure}
			\caption{$\lambda_2$-maximizer in $\mathcal{T}(n)$ when $n$ is even}
		\end{figure}
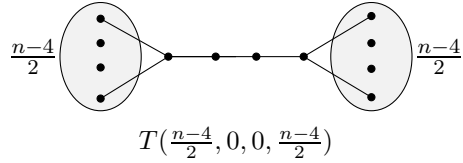
	\end{enumerate}
\end{theorem}

It is not difficult to show that for $T\in \mathcal{T}(n)$ ($n\ge3$), we have $\lambda_2(T)\geq 0$, and equality occurs if and only if $T$ is the star $K_{1, n-1}$. We mention here that the $\lambda_2$ extremization problem for star-like trees was considered by Bell and  Simi\'{c} \cite{Bell_Simic_1998}. 

An upper bound for $\lambda_2$ of general graphs was obtained by Powers \cite{Powers_1988}. For $G\in \mathcal{G}(n)$, it is easy to see that $\lambda_2(G)\ge -1$ and equality holds if and only if $G=K_n$. The problem of finding the extremal graphs in $\mathcal{G}(n)$ which maximize $\lambda_2$ was solved by Zhai, Lin and Wang \cite{Mingqing_2012}. 

\begin{theorem}[\cite{Mingqing_2012}]\label{thm:lambda_two_max_graphs}  Let $G$ be a $\lambda_2$-maximizer in $\mathcal{G}(n)$. 
	\begin{enumerate}[$(i)$]
		\item If $n$ is odd then $G$ is obtained from two disjoint copies of the complete graph $K_{\frac{n-1}{2}}$ by adding some edges from their vertices to an additional vertex $K_1$.
		\item If $n$ is even then $G$ is obtained from two disjoint copies of the complete graph $K_{\frac{n}{2}}$ by adding an edge between them.
	\end{enumerate}
\end{theorem}

If $G$ is a connected bipartite graph, then $\lambda_2(G)\ge 0$ and equality holds if and only if $G$ is a complete bipartite graph. Zhai, Lin and Wang \cite{Mingqing_2012} determined the extremal connected bipartite graphs with maximum $\lambda_2$. Chang \cite{Chang_1998} determined the extremal trees which minimize $\lambda_2$ in the class of trees on $n=2k$ vertices with a perfect matching. Chang \cite{Chang_1998} also posed a conjecture about the extremal trees which maximize $\lambda_2$ in this class. The conjecture was corrected and fully resolved by Guo and Tan \cite{Guo_Tan_2002, Guo_Tan_2004}. Brand, Guiduli and Imrich \cite{Brand_Guiduli_Imrich_2007} determined cubic graphs with maximum $\lambda_2$. Recently, the problem of maximizing $\lambda_2$ in the class of $n$-vertex connected outerplanar graphs was studied by Brooks, Gu, Hyatt, Linz, and Lu \cite{brooks_2023}. It is not difficult to show that for $n\ge 5$, the star $K_{1,n-1}$ is the unique $\lambda_2$-minimizer in this class.

\subsection{Our contribution}

In this paper, our main interest is to extremize $\lambda_2$ in the class $\mathcal{T}(n,d)$ of $n$-vertex trees with diameter $d$. We recall that the problem of maximizing $\lambda_1$ in $\mathcal{T}(n,d)$ was first solved by Tan, Guo and Qi \cite{Tan_Guo_Jian_2004} (see also \cite{Guo_Shao_2006, Simic_Zhou_2007, simic_caterpillar_2008}). They showed the following.

\begin{theorem}[\cite{Tan_Guo_Jian_2004}]\label{thm:lambda_one_max_trees_diameter} Let $T\in \mathcal{T}(n,d)$ be a $\lambda_1$-maximizer. Then $T=C(\lfloor\frac{d}{2}\rfloor, \lceil\frac{d}{2}\rceil, n-d-1)$.
	\begin{figure}[H]
		\centering
		\begin{tikzpicture}[scale=0.9]
			\draw [decorate, thick, decoration = {calligraphic brace, raise=5pt}] (-3,0) --  (-1,0) node[pos=0.5, above=10pt] {$P_{\lfloor\frac{d}{2}\rfloor}$};
			\draw [decorate, thick, decoration = {calligraphic brace, raise=5pt}] (1,0) --  (3,0) node[pos=0.5, above=10pt] {$P_{\lceil\frac{d}{2}\rceil}$};
			\draw [dashed] (-2,0)-- (-1,0);
			\draw  (-1,0)-- (0,0);
			\draw  (0,0)-- (1,0);
			\draw [dashed] (1,0)-- (2,0);
			\draw [rotate around={0:(-0.01,1)}, fill=gray!10] (-0.01,1) ellipse (0.7102313933670734cm and 0.3800376719802067cm);
			\draw  (0.5,1)-- (0,0);
			\draw  (-0.5,1)-- (0,0);
			\draw  (-3,0)-- (-2,0);
			\draw  (2,0)-- (3,0);
			\draw [fill=black] (-2,0) circle (1.5pt);
			\draw [fill=black] (-1,0) circle (1.5pt);
			\draw [fill=black] (0,0) circle (1.5pt);
			\draw (0,-0.4) node {$\lfloor \frac{d}{2} \rfloor$};
			\draw [fill=black] (1,0) circle (1.5pt);
			\draw [fill=black] (2,0) circle (1.5pt);
			\draw [fill=black] (-0.5,1) circle (1.5pt);
			\draw [fill=black] (0.5,1) circle (1.5pt);
			\draw (0,1.7) node {$n-d-1$};
			\draw [fill=black] (-3,0) circle (1.5pt);
			\draw (-3,-0.4) node {$0$};
			\draw [fill=black] (3,0) circle (1.5pt);
			\draw (3,-0.4) node {$d$};
		\end{tikzpicture}
		\caption{$C(\lfloor\frac{d}{2}\rfloor, \lceil\frac{d}{2}\rceil, n-d-1)$}
	\end{figure}
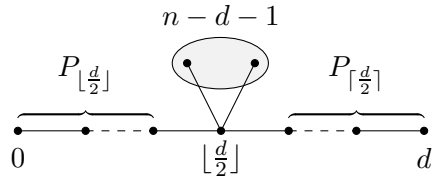
\end{theorem}
The problem of minimizing $\lambda_1$ in $\mathcal{T}(n,d)$ appears to be wide open, see \cite{van2007minimal, yuan2008minimal, cioabua2010asymptotic, lan2012graphs, lan2013diameters}. 

Note that if $d=n-1$ then $\mathcal{T}(n,d)=\{P_n\}$ and when $d=2$ then $\mathcal{T}(n,d)=\{K_{1,n-1}\}$. So we assume throughout that $n\ge 4$ and $3\le d\le n-2$. We first maximize $\lambda_2$ over the class $\mathcal{T}(n,d)$ and prove the following.

\begin{theorem}
	For every $n\ge d+1$, the $\lambda_2$-maximizer in $\mathcal{T}(n,d)$ is a caterpillar with at most two vertices of degree more than two as shown in \autoref{fig:lambda_2_maximizer}, where $k_1=\lfloor \frac{n-d-1}{2}\rfloor$ and $k_2=\lceil \frac{n-d-1}{2}\rceil$.  
\end{theorem}

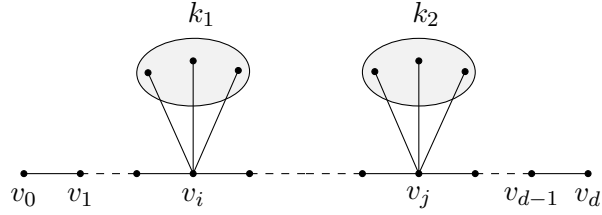
\begin{figure}[H]
	\centering
	\begin{tikzpicture}[scale=0.75]
		\draw [rotate around={1.4320961841646485:(-2,1.8)}, fill=gray!10] (-2,1.8) ellipse (1cm and 0.6cm);
		\draw [rotate around={0:(2,1.8)}, fill=gray!10] (2,1.8) ellipse (1cm and 0.6cm);
		
		\draw  (-4,0)-- (-5,0);
		\draw  [dashed](-4,0)-- (-3,0);
		\draw  (-3,0)-- (-2,0);
		\draw  (-2,0)-- (-1,0);
		\draw  [dashed] (-1,0)-- (0,0);
		\draw  [dashed](0,0)-- (1,0);
		\draw  (1,0)-- (2,0);
		\draw  (2,0)-- (3,0);
		\draw  [dashed](3,0)-- (4,0);
		\draw  (4,0)-- (5,0);
		
		\draw  (-2.8,1.79)-- (-2,0);
		\draw  (-1.2,1.83)-- (-2,0);
		\draw  (1.22,1.81)-- (2,0);
		\draw  (2.82,1.81)-- (2,0);
		\draw  (-2,2)-- (-2,0);
		\draw  (2,2)-- (2,0);
		
		\draw [fill=black] (-5,0) circle (1.5pt);
		\draw [fill=black] (-4,0) circle (1.5pt);
		\draw [fill=black] (-3,0) circle (1.5pt);
		\draw [fill=black] (-2,0) circle (1.5pt);
		\draw [fill=black] (-1,0) circle (1.5pt);
		\draw [fill=black] (1,0) circle (1.5pt);
		\draw [fill=black] (2,0) circle (1.5pt);
		\draw [fill=black] (3,0) circle (1.5pt);
		\draw [fill=black] (4,0) circle (1.5pt);
		\draw [fill=black] (5,0) circle (1.5pt);
		
		\draw (-5,-0.42) node {$v_0$};
		\draw (-4,-0.42) node {$v_1$};
		\draw (-2,-0.42) node {$v_i$};
		\draw (2,-0.42) node {$v_j$};
		\draw (4,-0.42) node {$v_{d-1}$};
		\draw (5,-0.42) node {$v_d$};
		
		\draw [fill=black] (-2.8,1.79) circle (1.5pt);
		\draw [fill=black] (-1.2,1.83) circle (1.5pt);
		\draw [fill=black] (1.22,1.81) circle (1.5pt);
		\draw [fill=black] (2.82,1.81) circle (1.5pt);
		\draw [fill=black] (-2,2) circle (1.5pt);
		\draw [fill=black] (2,2) circle (1.5pt);
		
		\draw (-1.85,2.83) node {$k_1$};
		\draw (2.13,2.83) node {$k_2$};
	\end{tikzpicture}
	\caption{$\lambda_2$-maximizer in $\mathcal{T}(n,d)$}
	\label{fig:lambda_2_maximizer}
\end{figure}

In fact, we determine the parameters of the $\lambda_2$-maximizers in $\mathcal{T}(n,d)$ for every possible case. They are stated precisely in \autoref{thm:lambda_two_max_trees_diameter} below.

\begin{theorem}\label{thm:lambda_two_max_trees_diameter} 
	Suppose that $n-d-1\geq 1$ and $d\ge 3$. Let $\lambda_2^*=\max\{\lambda_2(T): T\in \mathcal{T}(n,d)\}$ and let $T^*\in \mathcal{T}(n,d)$ be such that $\lambda_2(T^*)=\lambda_2^*.$ Then $T^*$ is obtained from a path $v_0 \ldots v_d$ by attaching $k_1$ leaves at 
	vertex $v_i$ and $k_2$ leaves at vertex $v_j$ (see \autoref{fig:lambda_2_maximizer}) where $k_1=\lfloor \frac{n-d-1}{2}\rfloor$, $k_2=\lceil \frac{n-d-1}{2}\rceil$ and $i<j$ are as follows:
	\begin{enumerate}[$(i)$]
		\item If $n-d-1$ is even, then the following holds:
		\begin{enumerate}[$(a)$]
			\item if $d=3,4$ then $i=1$ and $j=d-1$.
			\item if $d\ge 5$ is even then $i\in \{\lfloor \frac{d-2}{4}\rfloor,  \lceil\frac{d-2}{4}\rceil\}$ and $j\in \{d-\lfloor \frac{d-2}{4}\rfloor, d- \lceil\frac{d-2}{4}\rceil\}$, giving three nonisomorphic $\lambda_2$-maximizers.
			\item if $d\ge 5$ is odd then  $i=\lfloor \frac{d-1}{4}\rfloor$ and $j=d-\lfloor \frac{d-1}{4}\rfloor$.
		\end{enumerate}
		\item If $n-d-1\ge 3$ is odd, then the following holds:
		\begin{enumerate}[$(a)$]
			\item if $d=3,4$ then $i=1$ and $j=d-1$.
			\item if $d\ge 5$ and $n-d-1\ge 5$ then $i=\lfloor \frac{d-3}{2}\rfloor$ and $j=d-1$. 
			\item if $d\ge 5$ and $n-d-1=3$ then $i=\lfloor \frac{d-3}{2}\rfloor$ and $j=d-1$ if $d\le 10$, $i=\lfloor \frac{d-4}{2}\rfloor$ and $j=d-1$ if $d\ge 11$.
		\end{enumerate}
		\item If $n-d-1=1$ then $k_1=0$ and $k_2=1$. Hence, the position of $v_i$ is irrelevant. Moreover, we have $j=d-1$ if $d\le 10$, $j\in \{d-1,d-2\}$ if $d=11$, $j=d-2$ if $12\le d\le 22$, $j\in \{d-2,d-3\}$ if $d=23$ and $j=d-3$ if $d\ge 24$.
	\end{enumerate}
\end{theorem}

We also obtain an upper bound for the maximum $\lambda_2$ in the class $\mathcal{T}(n,d)$. This bound is never achieved for finite $n$ and $d$, but $\lambda_2^*$ can be arbitrarily close to it. 
\begin{corollary}\label{cor:lambda_two_max_trees_diameter_bound}
	Let $\lambda_2^*=\max\{\lambda_2(T): T\in \mathcal{T}(n,d)\}$. Then,
	\[\lambda_2^* < \sqrt{2+\sqrt{{\bigg(\frac{n-d}{2}\bigg)}^2+4}}. \]
\end{corollary}

Next, we determine the trees in $\mathcal{T}(n,d)$ which minimize $\lambda_2$ and obtain the following result.

\begin{theorem}\label{thm:lambda_two_min_trees_diameter}
	Let $\lambda_2^\#= \min \{\lambda_2(T): T\in \mathcal{T}(n,d)\}$ and let $T^\# \in \mathcal{T}(n,d)$ be such that $\lambda_2(T^\#)=\lambda_2^\#$. Then $T^\#$ looks like the tree shown in \autoref{fig:lambda_2_min}.
	\begin{figure}[H]
		\centering
		\begin{tikzpicture}[scale=0.8]
			\draw  (-3,0)-- (-2,0);
			\draw  [dashed](-2,0)-- (-1,0);
			\draw  (-1,0)-- (0,0);
			\draw  (0,0)-- (1,0);
			\draw  [dashed](1,0)-- (2,0);
			\draw  (2,0)-- (3,0);
			\draw [rotate around={90:(0,1.02)}, fill=gray!10] (0,1.02) ellipse (0.98cm and 0.8990550594930213cm);
			\draw [fill=black] (-3,0) circle (1.5pt);
			\draw [fill=black] (-2,0) circle (1.5pt);
			\draw [fill=black] (-1,0) circle (1.5pt);
			\draw [fill=black] (0,0) circle (1.5pt);
			\draw [fill=black] (1,0) circle (1.5pt);
			\draw [fill=black] (2,0) circle (1.5pt);
			\draw [fill=black] (3,0) circle (1.5pt);
			
			\draw (0,1.5) node {$B$};
			\draw (-3,-0.4) node {$v_0$};
			\draw (0,-0.4) node {$v_{\lfloor \frac{d}{2} \rfloor}$};
			\draw (3,-0.4) node {$v_d$};
		\end{tikzpicture}
		\caption{$\lambda_2$-minimizer in $\mathcal{T}(n,d)$}
		\label{fig:lambda_2_min}
	\end{figure}
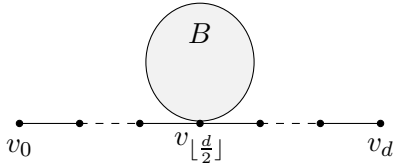
	\noindent
	Moreover, the following holds:
	\begin{enumerate}[$(i)$]
		\item $\lambda_1(B-v_{\lfloor \frac{d}{2} \rfloor})\le \lambda_2^\# \le \lambda_1(P_{\lceil \frac{d}{2}\rceil})$.
		\item If $d$ is even then $\lambda_2^\#= \lambda_1(P_{\frac{d}{2}})<2$ and $B-v_{\frac{d}{2}}$ is any disjoint union of Smith graphs (see \cite[Chapter 3]{Brouwer_Haemers_book}) with spectral radius at most $\lambda_2^\#$ (provided the diameter and order constraints on $T^\#$ are satisfied).
		\item If $d$ is odd then $\lambda_1(P_{\frac{d-1}{2}})< \lambda_2^\# < \lambda_1(P_{\frac{d+1}{2}})$ and $T^\#$ (in particular $B$) is a subdivision of a star with central vertex $v_{\frac{d-1}{2}}$. 
	\end{enumerate}
\end{theorem}

\subsection{Organization of the paper}

This paper is organized as follows. In \autoref{section:preliminaries}, we first recall some known results about characteristic polynomials of graphs. We then mention graph operations which increase (or decrease) $\lambda_1$ and conclude with structural results about $\lambda_2$ eigenvectors. In \autoref{section:spectral_center}, we revisit a notion of centrality in trees known as the \emph{spectral center}. This is a key ingredient in the proof of the main results. \autoref{section:caterpillars} is devoted to a standalone study of caterpillars of the type $C(\ell,r,k)$. We investigate its characteristic polynomial and obtain an upper bound for its spectral radius. In \autoref{section:lambda_2_max} and \ref{section:lambda_2_min} we prove Theorems \ref{thm:lambda_two_max_trees_diameter} and \ref{thm:lambda_two_min_trees_diameter}, respectively. We revisit \autoref{thm:lambda_two_max_trees} in \autoref{section:old_results}. We conclude with open problems for future work in \autoref{section:conclusion}.


\section{Preliminaries}\label{section:preliminaries}

Throughout this section, $G$ denotes a connected graph on $n$ vertices, and $A(G)$ denotes its adjacency matrix. We first recall a well-known result from Matrix Theory.
\begin{theorem}[Interlacing Theorem]
	Let $A$ be an $n\times n$ Hermitian matrix with eigenvalues $\lambda_1 \ge \cdots \ge \lambda_n$. Let $B$ be an $(n-1)\times (n-1)$ principal submatrix of $A$ with eigenvalues $\theta_1\ge \cdots \ge \theta_{n-1}$. Then $\lambda_i \ge \theta_i \ge \lambda_{i+1}$ for $1\le i \le n-1$.
\end{theorem}

In the following subsections, we mention some results about characteristic polynomials, $\lambda_1$ and $\lambda_2$, which are needed for the proofs of our main results.

\subsection{Characteristic polynomials}
For a graph $G$, its \emph{characteristic polynomial} is given by $\Phi(G,x):=\det(xI-A(G))$. The following lemma is easy to prove. For more details refer \cite[Section 2.2]{CRS_2010}.

\begin{lemma}\label{lemma:characteristic_G1G2} Let $(G_1,a)$ and $(G_2,b)$ be two disjoint rooted graphs with roots $a$ and $b$ respectively. 
	\begin{enumerate}[$(i)$]
		\item If $G=(G_1,a) \circ (G_2, b)$ then \[ \Phi(G,x)=\Phi(G_1,x)\Phi(G_2,x)-\Phi(G_1-a,x)\Phi(G_2-b,x).\]
		\item If $G=(G_1,a)\circ K_1 \circ (G_2,b)$ then \[ \Phi(G,x)=x\Phi(G_1,x)\Phi(G_2,x)-\Phi(G_1,x)\Phi(G_2-b,x)-\Phi(G_1-a,x)\Phi(G_2,x).\] 
		\item If $G$ is obtained from $(G_1,a)$ and $(G_2,b)$ by identifying $a$ and $b$ then 
		\[ \Phi(G,x)=\Phi(G_1-a,x)\Phi(G_2,x)+\Phi(G_1,x)\Phi(G_2-b,x)-x\Phi(G_1-a,x)\Phi(G_2-b,x).\]
	\end{enumerate}
\end{lemma}

For any vertex $a$ in $G$, the number of closed walks of length $k$ starting at $a$ is given by $A(G)^k_{a, a}$. It is known that the \emph{walk-generating function}
\[W(G, x) = I+xA(G) + x^2A(G)^2 + \cdots \]
converges to $(I-xA(G))^{-1}$ with radius of convergence $\lambda_1(G)^{-1}$.

\begin{lemma}[\cite{Godsil_1993}, Chapter 5]\label{lemma:ratio_fun}
	Let $W_{a, a}(G, x)$ be the generating function of closed walks at vertex $a$ in a graph $G$. Then for $x>\lambda_1(G)$,
	\[\frac{\Phi(G- a, x)}{\Phi(G, x)}=x^{-1} W_{a,a}(G, x^{-1})=(xI - A(G))^{-1}_{a,a}.\]
\end{lemma}

The above lemma applied to paths gives the following. For proof, see \cite[Lemma 14]{KMPZ_subdivision_2023}.
\begin{lemma}
	\label{lemma:ratio_lim}
	Let $P_\ell$ denote the $\ell$-vertex path. For $x\in(2,\infty)$, we have
	\[\lim_{\ell\to \infty}\frac{\Phi(P_{\ell-1}, x)}{\Phi(P_\ell, x)} = \frac{1}{2}(x-\sqrt{x^2-4}).\]
\end{lemma}

We also observe the following technical fact, which will be helpful later.

\begin{lemma}\label{lemma:deleted_vertex_char}
	Let $G\neq K_1$ be a connected graph and $a\in V(G)$. If $\lambda_1(G-a)<\lambda$ then $\Phi(G,\lambda)<\lambda\Phi(G-a,\lambda)$.
\end{lemma}

\begin{proof}
	If $\lambda_1(G)\ge \lambda$ then $\lambda_1(G)\ge \lambda >\lambda_1(G-a) \ge \lambda_2(G)$ where the last inequality follows by the Interlacing Theorem. It follows that $\Phi(G,\lambda)\le 0$ and $\Phi(G-a,\lambda)>0$. On the other hand, if $\lambda_1(G) < \lambda$ then $\Phi(G,\lambda)>0$. Consequently, using \autoref{lemma:ratio_fun},
	\[
	\frac{\Phi(G,\lambda) - \lambda\Phi(G-a,\lambda)}{\Phi(G,\lambda)}
	=1-W_{a,a}(G,\lambda^{-1}) 
	= -\big(\lambda^{-1}A(G)_{a,a}+\lambda^{-2}A^2(G)_{a,a}+\cdots)
	<0.
	\]
	The last inequality holds since $G\neq K_1$.
\end{proof} 

\subsection{Some results about $\lambda_1$}

Several graph perturbations are known in literature which increase (or decrease) $\lambda_1$. See \cite[Chapter 8]{CRS_2010} for a good overview of these perturbations. We recall some of them below. The following lemma shows that edge addition to a graph increases spectral radius. 
\begin{lemma} Let $G$ be a connected graph and $uv\notin E(G)$. Then $\lambda_1(G+uv)>\lambda_1(G)$. 
\end{lemma}

Under special conditions, edge rotations in a graph also increase the spectral radius. 
\begin{lemma}[Rotation Lemma]
	Let $u,v,w$ be vertices in a connected graph $G$ such that $v \sim u \nsim w$. Then, the graph $G'=G-uv+uw$ is said to be obtained from $G$ by \emph{rotating} the edge $uv$. Let $x$ be a positive $\lambda_1$-eigenvector of $G$. If $x_w\ge x_v$, then $\lambda_1(G')>\lambda_1(G)$.
\end{lemma}

Fix two vertices $u$ and $v$ in a graph $G$. Kelmans Operation in $G$ from $u$ to $v$ is defined as follows: replace each edge $ua$ with $va$ whenever $u\sim a\nsim v$ in $G$ to obtain a new graph $G'$. Note that $G$ and $G'$ have the same number of vertices and edges. Interestingly, we get the same graph $G'$ (up to isomorphism) when the roles of $u$ and $v$ are interchanged. The following was observed by Csikv\'{a}ri \cite{Csikvari_2009} and follows from the Rotation Lemma.  

\begin{lemma}[Kelmans Operation] \label{lemma:Kelmans} Let $u$ and $v$ be vertices in a connected graph $G$ such that $N(u)\backslash \{v\} \nsubseteq N(v)\backslash \{u\}$. If $G'$ is obtained from $G$ by applying Kelmans Operation from vertex $u$ to $v$ then $\lambda_1(G')>\lambda_1(G)$. 
\end{lemma}

We say a path $P=v_0 v_1\ldots v_s$ in a graph $G$ is an \emph{internal path} if $\deg(v_0)\ge 3$, $\deg(v_i)=2$ for $1\leq i \leq s-1$, and $\deg(v_s)\ge 3$. Let $uv$ be an edge in $G$. The graph $G'$ obtained from $G$ by \emph{contracting} the edge $uv$ is defined as follows: we replace each edge $ua$ with $va$ whenever $u\sim a \nsim v$, and then we delete the vertex $u$. The following is taken from \cite[Theorem 5]{Mckee_2018}.

\begin{lemma}[Contraction Lemma]\label{lemma:Contraction} Let $uv$ be an edge on an internal path in a connected graph $G$ such that $u$ and $v$ have no common neighbours. Obtain $G'$ from $G$ by contracting the edge $uv$. Then $\lambda_1(G')\ge \lambda_1(G)$ and equality holds if and only if $\lambda_1(G)=2$.
\end{lemma}

The above result is closely related to the Hoffman - Smith Theorem about subdividing an edge (see \cite{Hoffman_Smith_1979}, \cite[Chapter 8]{CRS_2010}).

We also recall here a result of Hofmeister \cite{hofmeister_1997}, who determined the tree in $\mathcal{T}(n)$ with second maximum $\lambda_1$. We state and give a short proof of this result below. 
\begin{proposition}\label{prop:second_largest_lambda_one_trees}
	Let $T\in \mathcal{T}(n)$ and $T\neq K_{1,n-1}$. Then 
	\[ \lambda_1(T)\le \sqrt{\frac{1}{2}\big(n-1+\sqrt{n^2-6n+13}\big)}\] and equality holds if and only if\/ $T$ is the caterpillar $T(n-3,1)$. 
	\begin{figure}[H]
		\centering
		\begin{tikzpicture}
			\draw [rotate around={90:(-2,0)}, fill=gray!10] (-2,0) ellipse (0.8cm and 0.6cm);
			\draw  (0,0)-- (-1,0);
			\draw  (-1,0)-- (-2,-0.62);
			\draw  (-1,0)-- (-2,0.56);
			\draw  (-1,0)-- (1,0);
			
			\draw [fill=black] (0,0) circle (1.5pt);
			\draw [fill=black] (-1,0) circle (1.5pt);
			\draw [fill=black] (-2,-0.62) circle (1.5pt);
			\draw [fill=black] (-2,0.56) circle (1.5pt);
			\draw [fill=black] (1,0) circle (1.5pt);
			\draw (-3.25,0) node {$n-3$};
			\draw [fill=black] (-2,0.2) circle (1.5pt);
			\draw [fill=black] (-2,-0.16) circle (1.5pt);
		\end{tikzpicture}
		\caption{$T(n-3,1)$}
	\end{figure}
\end{proposition}

\begin{proof}
	Since $T\neq K_{1,n-1}$, $T$ has diameter at least 3. If $T\neq T(n-3,1)$, then by Kelmans Operation, $\lambda_1(T)<\lambda_1(T(n-3,1))$. Using \autoref{lemma:characteristic_G1G2}, one can find the characteristic polynomial of $T(n-3,1)$ and then compute its largest root to obtain the stated upper bound on $\lambda_1(T)$. 
\end{proof}

\subsection{Some results about $\lambda_2$}\label{subsection:lambda2_known}

Subgraphs $H_1$ and $H_2$ of a graph $G$ are said to be \emph{strongly disjoint} if they are vertex disjoint and there are no edges in $G$ with one endpoint in $H_1$ and the other endpoint in $H_2.$

\begin{lemma}[Strongly Disjoint Subgraphs Lemma, \cite{Zhao_equiangular_2021, Chen_Hao_2021}]
	\label{lemma:strongly_disjoint}
	Suppose $H_1$ and $H_2$ are strongly disjoint subgraphs of a connected graph $G.$ If $\lambda_1(H_1) > \lambda_2(G)$ then $\lambda_1(H_2)<\lambda_2(G).$ In other words, if $\lambda_1(H_1)>\lambda$ and $\lambda_2(H_2)\ge \lambda$ for some $\lambda\in \mathbb{R}$ then $\lambda_2(G)>\lambda$.    
\end{lemma}

For a vector $y=(y_v)_{v\in V(G)} \in \mathbb{R}^{V(G)}$, define the \emph{support} of $y$ to be the set $S(y)=\{v \in V(G) : y_v\neq 0\}$. The \emph{positive support} of $y$ is $S^+(y) = \{v \in V(G) : y_v> 0\}$, the \emph{negative support} of $y$ is $S^-(y) = \{v \in V(y) : y_v< 0\}$, and the \emph{zero set} of $y$ is $S^0(y) = \{v \in V(G) : y_v= 0\}$.

Let $G[S(y)]$ denote the subgraph of $G$ induced by the vertices in $S(y).$ A (connected) component of $G[S(y)]$ is called a \emph{$y$-component}. Similarly, one can define $G[S^*(y)]$ and $y^*$-component for $*\in \{+,-,0\}$. Note that if $y$ is the $\lambda_2$-eigenvector of a connected graph $G$, then $y$ is orthogonal to the positive $\lambda_1$-eigenvector of $G$ and therefore $S^+(y)$ and $S^-(y)$ are non-empty.  

As an application of the above lemma, Chen and Hao \cite{Chen_Hao_2021} obtained the following results. We mention here that closely related results (perhaps not in the exact same form) can be found in the works of Fiedler \cite{Fiedler_1975} and Powers \cite{Powers_1988}. See also \cite[Chapter 8]{CRS_2010} and \cite[Chapter 5]{Brouwer_Haemers_book}.

\begin{lemma}[\cite{Chen_Hao_2021}]\label{lemma:component_inequality}
	Let $y$ be a $\lambda_2$-eigenvector of $G$. If $H$ is a $y^+$-component ($y^-$-component), then $\lambda_1(H) \geq \lambda_2$ and equality holds if and only if $N(H)\cap S^-(y)=\emptyset$ ($N(H)\cap S^+(y)=\emptyset$).
\end{lemma}

\begin{lemma}[\cite{Chen_Hao_2021}]\label{lemma:component_connect}
	Let $y$ be a $\lambda_2$-eigenvector of $G.$
	\begin{enumerate}[$(i)$]
		\item If $G[S(y)]$ is connected then $G[S^+(y)]$ and $G[S^-(y)]$ are connected.
		\item If $G[S(y)]$ is disconnected then every $y$-component is a $y^+$-component or a $y^-$-component. Moreover, if $H$ is any $y$-component then $\lambda_1(H)=\lambda_2.$ 
		\item For $v\in S^0(y),$ if $N(v)\cap S(y)\neq \emptyset,$ then $v$ is adjacent to every $y^+$-component and every $y^-$-component.
	\end{enumerate} 
\end{lemma}

\begin{corollary}\label{cor:two_components}
	Let $y$ be a $\lambda_2$-eigenvector of $G$ with minimal support (i.e. $|S(y)|$ is minimal among all $\lambda_2$-eigenvectors). Then, $G$ has a unique $y^+$-component, and a unique $y^-$-component. 
\end{corollary}

A proof of \autoref{cor:two_components} can also be found in \cite[Chapter 5]{Brouwer_Haemers_book}.

We will often use the following observation without mention in this paper.

\begin{lemma}
	Let $G$ and $G'$ be two graphs on $n$-vertices.
	\begin{enumerate}[$(i)$]
		\item If $\lambda_1(G')>\lambda_2(G)$ and $\Phi(G', \lambda_2(G))>0$ then $\lambda_2(G')>\lambda_2(G)$.
		\item If $\lambda_3(G')<\lambda_2(G)$ and $\Phi(G', \lambda_2(G))<0$ then $\lambda_2(G')<\lambda_2(G)$.
	\end{enumerate}
\end{lemma}


\section{Spectral center of a tree}\label{section:spectral_center}

We now turn to a result about the \emph{spectral center}, a notion of centrality in trees defined using their second largest eigenvalue. We first show the following.

\begin{theorem}\label{thm:spectral_center}
	Let $T$ be a tree on $n\ge 2$ vertices. Let $y$ be a $\lambda_2(T)$-eigenvector with minimal support. Then there exist rooted subtrees $(H_1,a)$ and $(H_2,b)$ such that $H_1$ is the unique $y^+$-component and $H_2$ is the unique $y^-$-component of $T$. Furthermore, exactly one of the following holds. 
	\begin{enumerate}[$(i)$]
		\item If $S^0(y)\neq \emptyset$ then there is a unique vertex $v\in S^0(y)$ such that $T'=(H_1,a)\circ v \circ (H_2,b)$ is a subtree of $T$, $V(T)\backslash V(T')\subseteq S^0(y)$ and $T[S^0(y)]$ is connected. Moreover, 
		\[\lambda_1(H_1)=\lambda_1(H_2)=\lambda_2(T).\]  
		The vertex $v$ is called the \emph{spectral vertex} of $T$.
		\item If $S^0(y)=\emptyset$ then $T=(H_1,a)\circ (H_2,b)$ and
		\[\lambda_1(H_1-a) < \lambda_2(T) < \lambda_1(H_1) \text{ and } \lambda_1(H_2-b) < \lambda_2(T) < \lambda_1(H_2).\]
		The edge $ab$ is called the \emph{spectral edge} of $T$.
	\end{enumerate}
\end{theorem} 

\begin{proof}
	By \autoref{cor:two_components}, $T$ has a unique $y^+$-component (say $H_1$) and a unique $y^-$-component (say $H_2$). We make the following cases:
	
	\textbf{Case 1:} $S^0(y)\neq \emptyset$.~
	Since $T$ is connected, there exists a vertex $v\in S^0(y)$ such that $N(v)\cap S(y)\neq \emptyset$. By \autoref{lemma:component_connect}, $N(v)\cap V(H_i)\neq \emptyset$ for $i=1,2$. Since $T$ is acyclic, this vertex $v$ is unique, and $v$ has a unique neighbour (say $a$) in $H_1$ and a unique neighbour (say $b$) in $H_2$. Then $(H_1,a)$ and $(H_2,b)$ are the required rooted subtrees of $T$. Since, $H_1$ and $H_2$ are strongly disjoint, we conclude that $\lambda_1(H_1)=\lambda_1(H_2)=\lambda_2(T)$ by \autoref{lemma:component_inequality}. Also, any vertex outside of $T'$ in $T$ has to lie in $S^0(y)$ by the choice of $H_1$ and $H_2$. Moreover, any vertex in $S^0(y)$ other than $v$ cannot have any neighbour in $H_1$ or $H_2$; hence $S^0(y)$ induces a subtree rooted at $v$ in $T$.  
	
	\textbf{Case 2:} $S^0(y)=\emptyset$.~
	Then $T[S(y)]=T$ is connected. By \autoref{lemma:component_connect}, $H_1$ and $H_2$ are connected. Since, $T$ is a tree, there exists a unique edge $ab$ in $T$ joining $H_1$ and $H_2$ where $a$ lies in $H_1$ and $b$ lies in $H_2$. Thus, $T=(H_1,a)\circ (H_2,b)$. Moreover, by \autoref{lemma:component_inequality}, 
	$\lambda_2(T) < \lambda_1(H_1) \text{ and }  \lambda_2(T) < \lambda_1(H_2)$.
	Using \autoref{lemma:strongly_disjoint} and the above inequalities, we conclude that
	$\lambda_1(H_1-a) < \lambda_2(T)$ and $\lambda_1(H_2-b) < \lambda_2(T)$.
	This completes the proof.
\end{proof}

A result similar in flavour to the above theorem was first proved by Neumaier \cite{neumaier_second_1982}. He showed that every tree $T$ is a $\lambda$-twin or $\lambda$-trivial for any $\lambda\ge \lambda_2(T)$, which correspond to the situations when $T$ has a spectral edge or a spectral vertex, respectively. The following immediate corollary can be found in \cite[Chapter 5]{Brouwer_Haemers_book}. 

\begin{corollary}\label{cor:specral_center}
	Every tree $T$ on $n\ge 2$ vertices has a unique minimal subtree $Y$ on at most two vertices such that $\lambda_1(T-Y)\leq \lambda_2(T)$. Moreover, if $|V(Y)|=1$ then $\lambda_1(T-Y)=\lambda_2(T)$, and if $|V(Y)|=2$ then $\lambda_1(T-Y)< \lambda_2(T)$.
\end{corollary}

The subtree $Y$ from the above corollary is called the \emph{spectral center} of the tree $T$. 


\section{On caterpillars of type $C(\ell,r,k)$}\label{section:caterpillars}

In this section, we investigate caterpillars of the type $C(\ell,r,k)$ (for definition, see \autoref{section:introduction}). We will see in \autoref{section:lambda_2_max} that $\lambda_2$-maximizers in $\mathcal{T}(n,d)$ are obtained by taking two caterpillars of this type and joining them by an edge. Using \autoref{lemma:characteristic_G1G2}, we get the following. 

\begin{lemma}\label{lemma:characteristic_central_caterpillar}
	Let $T=C(\ell,r,k)\in \mathcal{T}(\ell+r+k+1, \ell+r)$. 
	\begin{enumerate}[$(i)$]
		\item If $k=1$ then $\Phi(T,x)=x^{k-1}\big( x\Phi(P_{\ell+r+1},x)-k\Phi(P_\ell,x)\Phi(P_r,x)\big)$.
		\item If $k=0$ then $\Phi(T,x)=\Phi(P_{\ell+r+1},x)$.
	\end{enumerate} 
\end{lemma}

Here $P_\ell$ denotes a path on $\ell$ vertices. Note that if $\ell=0$ then $\Phi(P_\ell,x)=1$. 

We say (roughly speaking) that $C(\ell,r,k)$ is a balanced caterpillar if $|\ell-r|$ is small. The following lemma compares the characteristic polynomial of a more balanced caterpillar with that of a less balanced caterpillar in the interval $(2,\infty)$. We work with this interval because the spectral radius of a path is at most 2. 

\begin{lemma}\label{lemma:characteristic_unbalanced_central_caterpillar} Let $\ell\ge r\ge 1$ and $k\ge 1$. Then for all $x> 2$, \[\Phi(C(\ell+1,r-1,k),x)>\Phi(C(\ell,r,k),x).\] 
\end{lemma}

\begin{proof}
	By \autoref{lemma:characteristic_central_caterpillar},
	\[\Phi(C(\ell+1,r-1,k),x)-\Phi(C(\ell,r,k),x)
	= kx^{k-1}\big[\Phi(P_{\ell},x) \Phi(P_r, x) - \Phi(P_{\ell+1}, x)\Phi(P_{r-1}, x)\big],
	\]
	which is positive for $x> 2$ if and only if 
	\[\frac{\Phi(P_{\ell},x)}{\Phi(P_{\ell+1},x)} > \frac{\Phi(P_{r-1},x)}{\Phi(P_{r},x)},\]
	which, by \autoref{lemma:ratio_fun}, occurs if and only if $W_{v,v}(P_{\ell + 1}, x^{-1})>W_{v,v}(P_r, x^{-1})$,
	where $v$ denotes the first vertex of $P_{\ell + 1}$ and $P_r$. Since $\ell+1>r$, the path $P_r$ is a proper induced subgraph of $P_{\ell+1}$, and the above inequality holds. 
\end{proof}

Next, we find an upper bound for the spectral radius of $C(\ell,r,k)$.

\begin{proposition}\label{prop:lambda1_central_caterpillar} 
	Let $\ell,r,k$ be nonnegative integers. Then $\lambda_1(C(\ell,r,k))< \sqrt{\sqrt{k^2+4}+2}$.
\end{proposition}

\begin{proof}
	If $k=0$ then $C(\ell,r,k)=P_{\ell+r+1}$ and the claim holds. So, assume $k\ge 1$. We will show that 
	\[\lim_{\ell\to \infty}\lambda_1(C(\ell,\ell,k))=\sqrt{\sqrt{k^2+4}+2}.\] 
	By \autoref{lemma:characteristic_central_caterpillar}, $x>0$ is an eigenvalue for $C(\ell,\ell,k)$ if and only if 
	\[
	\frac{k}{x}=\frac{\Phi(P_{2\ell+1},x)}{\Phi(P_\ell,x)^2}
	=\frac{2\Phi(P_{\ell+1},x)\Phi(P_\ell,x)-x\Phi(P_\ell,x)^2}{\Phi(P_\ell,x)^2}
	=\frac{2\Phi(P_{\ell+1},x)}{\Phi(P_\ell,x)}-x,
	\]
	where the second equality holds by \autoref{lemma:characteristic_G1G2}. Using \autoref{lemma:ratio_lim}, for all $x>2$, 
	\[ \lim_{\ell\to \infty}\frac{\Phi(P_{\ell+1},x)}{\Phi(P_\ell,x)}=\frac{1}{2}(x+\sqrt{x^2-4}).\]
	Thus $\lim_{\ell\to \infty}\lambda_1(C(\ell,\ell,k))$ is equal to the largest root (provided it is bigger than 2) of 
	\begin{equation}\label{eq:lambda1_caterpillar}
		\frac{k}{x}=\sqrt{x^2-4}.
	\end{equation}
	The largest root of \eqref{eq:lambda1_caterpillar} is $\sqrt{\sqrt{k^2+4}+2}$ and it is bigger than 2 whenever $k\ge 1$.
\end{proof}

The following proposition suggests that $k$ (i.e. the number of pendant vertices not on the main path in $C(\ell,r,k)$) is determined by the spectral radius of $C(\ell,r,k)$.

\begin{proposition}\label{prop:central_caterpillar_isomorphism}
	Let $H_1=C(\ell_1,r_1,k_1)$ and $H_2=C(\ell_2,r_2,k_2)$ where $\ell_i\ge 1$ and $r_i\ge 1$ for $i=1,2$. If $\lambda_1(H_1)=\lambda_1(H_2)>2$ then $k_1=k_2$.
\end{proposition}

\begin{proof}
	Suppose to the contrary that $k_1<k_2$. Since $\lambda_1(H_1)>2$, $H_1$ is not a path and hence $k_1\ge 1$. If $k_1\ge 2$, then by \autoref{prop:lambda1_central_caterpillar},
	\[ \lambda_1(H_1)< \sqrt{\sqrt{k_1^2+4}+2} < \sqrt{k_1+3}\le \sqrt{k_2+2}\le \lambda_1(H_2).\]
	This is a contradiction. Suppose now that $k_1=1$. Then $k_2\ge 2$. Since $\lambda_1(H_2)>2$, $H_2\ncong K_{1,4}$. So, $C(1,2,2)$ or $C(1,1,3)$ is a subgraph of $H_2$. By \autoref{prop:lambda1_central_caterpillar},
	\[ \lambda_1(H_1)<\sqrt{\sqrt{5}+2} < \sqrt{\frac{5+\sqrt{13}}{2}}=\min\{\lambda_1(C(1,2,2)),\lambda_1(C(1,1,3))\}\le \lambda_1(H_2),\]
	a contradiction. This completes the proof.
\end{proof}

We conclude this section with bounds for $\lambda_2$ of the caterpillar $C(\lfloor\frac{d}{2}\rfloor, \lceil\frac{d}{2}\rceil, n-d-1)$. This will be useful in \autoref{section:lambda_2_min} when we minimize $\lambda_2$ in $\mathcal{T}(n,d)$.

\begin{proposition}\label{prop:caterpillar_lambda_2} Let $n\ge d+1$ and $T=C(\lfloor\frac{d}{2}\rfloor, \lceil\frac{d}{2}\rceil, n-d-1)$ be a caterpillar on the path $v_0 \ldots v_d$. 
	\begin{enumerate}[$(i)$]
		\item If $d$ is even then $\lambda_2(T)=\lambda_1(P_{\frac{d}{2}})$.
		\item If $d$ is odd then $\lambda_1(P_{\frac{d-1}{2}})<\lambda_2(T)<\lambda_1(P_{\frac{d+1}{2}})$.
	\end{enumerate}
\end{proposition}

\begin{proof}
	Using the Interlacing Theorem, $\lambda_1(P_{\lfloor\frac{d}{2}\rfloor})\le \lambda_2(T)\le \lambda_1(P_{\lceil\frac{d}{2}\rceil})$. Claim $(i)$ is now immediate. So suppose $d$ is odd. Let $(H_1,a)$ and $(H_2,b)$ be the rooted subtrees of $T$ as in \autoref{thm:spectral_center}. Clearly, both $a$ and $b$ lie on the path $v_0,\ldots,v_d$. If $T$ has a spectral vertex then one of the $H_i$ (say $H_1$) is a subgraph of $P_{\frac{d-1}{2}}$ and the other one ($H_2$) contains $P_{\frac{d+1}{2}}$ as a subgraph. But then $\lambda_1(H_1)\neq \lambda_1(H_2)$. This contradicts case $(i)$ of \autoref{thm:spectral_center}. So $T$ has a spectral edge, namely $ab$. If $ab\neq v_{\lfloor\frac{d}{2}\rfloor}v_{\lfloor\frac{d}{2}\rfloor+1}$, then again $H_1$ (say) is a subgraph of $P_{\frac{d-1}{2}}$. But then $\lambda_2(T)\ge \lambda_1(P_{\frac{d-1}{2}})\ge \lambda_1(H_1)$. This contradicts case $(ii)$ of \autoref{thm:spectral_center}. We conclude, without loss of generality, that $a=v_{\lfloor\frac{d}{2}\rfloor}$ and $b=v_{\lfloor\frac{d}{2}\rfloor+1}$. Hence, $\lambda_1(P_{\frac{d-1}{2}})=\lambda_1(H_1-a)<\lambda_2(T)<\lambda_1(H_1)=\lambda_1(P_{\frac{d+1}{2}})$. This proves claim $(ii)$.
\end{proof}


\section{$\lambda_2$-maximization in $\mathcal{T}(n,d)$}\label{section:lambda_2_max}

In this section, we prove our first main result, \autoref{thm:lambda_two_max_trees_diameter}. We assume that $n\ge 4$ and $3\le d\le n-2$. Let $\lambda_2^*=\max\{\lambda_2(T): T\in \mathcal{T}(n,d)\}$ and let $T^*\in \mathcal{T}(n,d)$ be a tree such that $\lambda_2(T^*)=\lambda_2^*.$ Since $d\ge 3$, $T^*$ is not a star and so $\lambda_2^*>0$. Let $(H_1,a)$ and $(H_2,b)$ be the rooted subtrees of $T^*$ as given in \autoref{thm:spectral_center}. Since, $\lambda_2^*>0$, $H_1$ and $H_2$ are not isomorphic to $K_1$.  Next, we prove a series of claims about the structure of $T^*$.

\begin{claim}
	$T^*=(H_1,a)\circ (H_2,b)$ or $T^*=(H_1,a)\circ v \circ (H_2,b)$.  
\end{claim}

\begin{proof}
	By \autoref{thm:spectral_center}, we only need to consider the spectral vertex case. Let $v$ be the spectral vertex in $T^*$. Suppose to the contrary, there exists $w\in V(T^*)\backslash V((H_1, a)\circ v \circ (H_2, b))$. Obtain $H_1'$ from $H_1$ by attaching the vertex $w$ to $a$. Let $T'\in \mathcal{T}(n,d)$ be such that $(H_1',a)\circ v\circ (H_2,b)$ is a subgraph of $T'$. Now $\lambda_1(H_1')>\lambda_1(H_1)=\lambda_2^*$. Since, $H_1'$ and $H_2$ are strongly disjoint in $T'$, by \autoref{lemma:strongly_disjoint}, $\lambda_2(T')>\lambda_2^*$. This contradicts the choice of $T^*$.
\end{proof}

\begin{claim} 
	There is a longest path of length $d$ in $T^*$ that passes through $a$ and $b$. 
\end{claim} 

\begin{proof} 
	Suppose not. Then, without loss of generality, we can assume that there is a longest path of length $d$, which lies entirely in $H_1$. Since $\lambda_2^*>0$, we have $H_2\neq K_1$ and so $d\geq 6$. Consider 
	$T'=(H_1',a')\circ K_1 \circ (H_2', b')$ where $(H_1',a')= (C(\lfloor\frac{d-3}{2}\rfloor, \lceil\frac{d-3}{2}\rceil, |H_1|-d+1), v_{d-3})$ is a rooted caterpillar on a path $v_0 \ldots v_{d-3}$ and $(H_2',b')=(K_{1,|H_2|-1},b')$ is a rooted star with $b'$ being the central vertex of the star. Clearly, $T'\in \mathcal{T}(n,d)$. Also,    
	\begin{equation}\label{eq:Claim_1_1}
		\lambda_1(H_2)\le \lambda_1(H_2').    
	\end{equation}
	Let $F_1=T(r_1,\ldots,r_{d-3})\in \mathcal{T}(|H_1|, d-2)$ where 
	\[ r_i= 
	\begin{cases}
		|H_1|-d & \text{ if } i= \lfloor\frac{d}{2}-1\rfloor,\\
		1 & \text{ if }i=1,\\
		1 & \text{ if }i \in \{\lfloor\frac{d}{2}\rfloor, d-3\} \text{ and }d>6,\\
		2 & \text{ if }i=d-3 \text{ and }d=6,\\
		0 & \text{ otherwise}.
	\end{cases}\]
	Now observe that
	\begin{equation}\label{eq:Claim_1_2}
		\lambda_1(H_1)  \le  \lambda_1(C(\bigg\lfloor\frac{d}{2}\bigg\rfloor, \bigg\lceil\frac{d}{2}\bigg\rceil, |H_1|-d-1))
		< \lambda_1(F_1)
		\le \lambda_1(H_1'), 
	\end{equation}
	where the first, the second and the third inequality hold by \autoref{thm:lambda_one_max_trees_diameter}, Kelmans Operation and Contraction Lemma, respectively. Since, $H_1'$ and $H_2'$ are strongly disjoint in $T'$, by \eqref{eq:Claim_1_1},  \eqref{eq:Claim_1_2}, \autoref{thm:spectral_center} and \autoref{lemma:strongly_disjoint}, we have that $\lambda_2(T')>\lambda_2^*$, a contradiction.
\end{proof}

Let $Q$ be a longest path in $T^*$ passing through $a$ and $b$. Let $Q_1=u_0\ldots u_{s_1}$ (where $u_{s_1} = a$) and $Q_2 = v_0\ldots v_{s_2}$ (where $v_0=b$) be the subpaths of $Q$ in $H_1$ and $H_2$ respectively. Note that $Q=(Q_1,a)\circ (Q_2,b)$ or $Q=(Q_1,a)\circ K_1 \circ (Q_2,b)$.

\begin{claim}
	$(H_1,a)$ and $(H_2,b)$ are caterpillars on paths $Q_1$ and $Q_2$ respectively.
\end{claim}

\begin{proof}
	Without loss of generality, suppose that $H_1$ is not a caterpillar. Then there exists a non-leaf $u$ in $H_1$ and not on $Q_1$ with a neighbor $u_i$ ($2\le i \le s_1$) on $Q_1$. Note here that $i\ge 2$ because the diameter of $T^*$ is $d$. Apply Kelmans Operation from $u$ to $u_{i-1}$ to obtain $(\Tilde{H}_1, a)$. Observe that in $(\Tilde{H}_1,a)$ the vertices $u_{i-1}$ and $u_i$ both have degrees at least 3. Contract the edge $u_{i-1}u_i$ and add a new leaf $a'$ at $a$ in $(\Tilde{H}_1,a)$ to obtain the rooted subtree $(H_1',a')$. Now consider the tree $T'$ obtained by replacing $(H_1,a)$ in $T^*$ with $(H_1',a')$. Then, $\lambda_1(H_1) < \lambda_1(\Tilde{H}_1) \le \lambda_1(H_1'-a')$,
	where the first inequality holds by Kelmans Operation and the second inequality holds by Contraction Lemma. Since $H_1'-a'$ and $H_2$ are strongly disjoint in $T'$, by \autoref{thm:spectral_center} and \autoref{lemma:strongly_disjoint}, we have that $\lambda_2(T')>\lambda_2^*$. This is a contradiction. 
\end{proof}

The following claim shows that $H_1$ and $H_2$ are caterpillars of type $C(\ell,r,k)$.

\begin{claim}
	$(H_1,a)=(C(\ell_1,r_1,k_1),u_{s_1})$ and $(H_2,b)=(C(\ell_2,r_2,k_2),v_0)$ for some non-negative integers $\ell_i, r_i, k_i$ such that $\ell_i+r_i+k_i+1=|H_i|$ and $\ell_i+r_i=s_i$ for $i=1,2$. 
\end{claim}

\begin{proof}
	Suppose that $(H_1, a)$ is not of the form $C(\ell_1,r_1,k_1)$. Then there are vertices $u_i$ and $u_j$ ($1\le i < j \le s_1$) on $Q_1$ with degree (in $T^*$) at least 3 and all the vertices on $Q_1$ between $u_i$ and $u_j$ have degree 2. Obtain $(H_1',a')$ from $(H_1,a)$ by contracting the edge $u_iu_{i+1}$ and then adding a new leaf $a'$ at $a$. Obtain $T'$ from $T^*$ by replacing $(H_1,a)$ with $(H_1',a')$. By Contraction Lemma, $\lambda_1(H_1) \le \lambda_1(H_1'-a')<\lambda_1(H_1')$. If $T^*$ has a spectral vertex, then $\lambda_2^*=\lambda_1(H_1)<\lambda_1(H_1')$. Since $H_1'$ and $H_2$ are strongly disjoint in $T'$, by \autoref{lemma:strongly_disjoint}, $\lambda_2(T')>\lambda_2^*$, a contradiction.
	If $T^*$ has spectral edge $ab$ then $\lambda_2^*<\lambda_1(H_1)\le \lambda_1(H_1'-a').$ Since $H_1'-a'$ and $H_2$ are strongly disjoint in $T'$, again by \autoref{lemma:strongly_disjoint}, $\lambda_2(T')>\lambda_2^*$, a contradiction.
\end{proof}

\begin{figure}[H]
	\begin{subfigure}{1.0\textwidth}
		\centering
		\begin{tikzpicture}[scale=0.7]
			\draw [rotate around={1.4320961841646485:(-4,1.8)}, fill=gray!10] (-4,1.8) ellipse (1cm and 0.6cm);
			\draw [rotate around={0:(4,1.8)}, fill=gray!10] (4,1.8) ellipse (1cm and 0.6cm);
			\draw [decorate, thick, decoration = {calligraphic brace, raise=5pt}] (-7,0) --  (-5,0) node[pos=0.5, above=10pt] {$P_{\ell_1}$};
			\draw [decorate, thick, decoration = {calligraphic brace, raise=5pt}] (-3,0) --  (-1,0) node[pos=0.5, above=10pt] {$P_{r_1}$};
			\draw [decorate, thick, decoration = {calligraphic brace, raise=5pt}] (1,0) --  (3,0) node[pos=0.5, above=10pt] {$P_{\ell_2}$};
			\draw [decorate, thick, decoration = {calligraphic brace, raise=5pt}] (5,0) --  (7,0) node[pos=0.5, above=10pt] {$P_{r_2}$};
			
			\draw  (-6,0)-- (-7,0);
			\draw  [dashed](-6,0)-- (-5,0);
			\draw  (-5,0)-- (-4,0);
			\draw  (-4,0)-- (-3,0);
			\draw  [dashed] (-3,0)-- (-2,0);
			\draw  (-2,0)-- (-1,0);
			\draw  (-1,0)-- (1,0);
			\draw  (1,0)-- (2,0);
			\draw  [dashed] (2,0)-- (3,0);
			\draw  (3,0)-- (4,0);
			\draw  (4,0)-- (5,0);
			\draw  [dashed](5,0)-- (6,0);
			\draw  (6,0)-- (7,0);
			
			\draw  (-4.8,1.79)-- (-4,0);
			\draw  (-3.2,1.83)-- (-4,0);
			\draw  (3.22,1.81)-- (4,0);
			\draw  (4.82,1.81)-- (4,0);
			\draw  (-4,2)-- (-4,0);
			\draw  (4,2)-- (4,0);
			
			\draw [fill=black] (-7,0) circle (1.5pt);
			\draw [fill=black] (-6,0) circle (1.5pt);
			\draw [fill=black] (-5,0) circle (1.5pt);
			\draw [fill=black] (-4,0) circle (1.5pt);
			\draw [fill=black] (-3,0) circle (1.5pt);
			\draw [fill=black] (-2,0) circle (1.5pt);
			\draw [fill=black] (-1,0) circle (1.5pt);
			\draw [fill=black] (1,0) circle (1.5pt);
			\draw [fill=black] (2,0) circle (1.5pt);
			\draw [fill=black] (3,0) circle (1.5pt);
			\draw [fill=black] (4,0) circle (1.5pt);
			\draw [fill=black] (5,0) circle (1.5pt);
			\draw [fill=black] (6,0) circle (1.5pt);
			\draw [fill=black] (7,0) circle (1.5pt);

			\draw [fill=black] (-4.8,1.79) circle (1.5pt);
			\draw [fill=black] (-3.2,1.83) circle (1.5pt);
			\draw [fill=black] (3.22,1.81) circle (1.5pt);
			\draw [fill=black] (4.82,1.81) circle (1.5pt);
			\draw [fill=black] (-4,2) circle (1.5pt);
			\draw [fill=black] (4,2) circle (1.5pt);

			\draw (-7,-0.43) node {$u_0$};
			\draw (-1,-0.43) node {$u_{s_1}=a$};
			\draw (1,-0.43) node {$v_0=b$};
			\draw (7,-0.43) node {$v_{s_2}$};
			\draw (-3.84,2.84) node {$k_1$};
			\draw (4.14,2.84) node {$k_2$};
			\draw (-8,0) node {$T^* \hspace{2pt}= $};
		\end{tikzpicture}
		\caption{if $T^*$ has a spectral edge $ab$}
	\end{subfigure}\\[0.2cm]
	\begin{subfigure}{1.0\textwidth}
		\centering
		\begin{tikzpicture}[scale=0.7]
			\draw [rotate around={1.4320961841646485:(-4,1.8)}, fill=gray!10] (-4,1.8) ellipse (1cm and 0.6cm);
			\draw [rotate around={0:(4,1.8)}, fill=gray!10] (4,1.8) ellipse (1cm and 0.6cm);
			\draw [decorate, thick, decoration = {calligraphic brace, raise=5pt}] (-7,0) --  (-5,0) node[pos=0.5, above=10pt] {$P_{\ell_1}$};
			\draw [decorate, thick, decoration = {calligraphic brace, raise=5pt}] (-3,0) --  (-1,0) node[pos=0.5, above=10pt] {$P_{r_1}$};
			\draw [decorate, thick, decoration = {calligraphic brace, raise=5pt}] (1,0) --  (3,0) node[pos=0.5, above=10pt] {$P_{\ell_2}$};
			\draw [decorate, thick, decoration = {calligraphic brace, raise=5pt}] (5,0) --  (7,0) node[pos=0.5, above=10pt] {$P_{r_2}$};
			
			\draw  (-6,0)-- (-7,0);
			\draw  [dashed](-6,0)-- (-5,0);
			\draw  (-5,0)-- (-4,0);
			\draw  (-4,0)-- (-3,0);
			\draw  [dashed] (-3,0)-- (-2,0);
			\draw  (-2,0)-- (-1,0);
			\draw  (-1,0)-- (0,0);
			\draw  (0,0)-- (1,0);
			\draw  (1,0)-- (2,0);
			\draw  [dashed] (2,0)-- (3,0);
			\draw  (3,0)-- (4,0);
			\draw  (4,0)-- (5,0);
			\draw  [dashed](5,0)-- (6,0);
			\draw  (6,0)-- (7,0);
			
			\draw  (-4.8,1.79)-- (-4,0);
			\draw  (-3.2,1.83)-- (-4,0);
			\draw  (3.22,1.81)-- (4,0);
			\draw  (4.82,1.81)-- (4,0);
			\draw  (-4,2)-- (-4,0);
			\draw  (4,2)-- (4,0);
			
			\draw [fill=black] (-7,0) circle (1.5pt);
			\draw [fill=black] (-6,0) circle (1.5pt);
			\draw [fill=black] (-5,0) circle (1.5pt);
			\draw [fill=black] (-4,0) circle (1.5pt);
			\draw [fill=black] (-3,0) circle (1.5pt);
			\draw [fill=black] (-2,0) circle (1.5pt);
			\draw [fill=black] (-1,0) circle (1.5pt);
			\draw [fill=black] (0,0) circle (1.5pt);
			\draw [fill=black] (1,0) circle (1.5pt);
			\draw [fill=black] (2,0) circle (1.5pt);
			\draw [fill=black] (3,0) circle (1.5pt);
			\draw [fill=black] (4,0) circle (1.5pt);
			\draw [fill=black] (5,0) circle (1.5pt);
			\draw [fill=black] (6,0) circle (1.5pt);
			\draw [fill=black] (7,0) circle (1.5pt);

			\draw [fill=black] (-4.8,1.79) circle (1.5pt);
			\draw [fill=black] (-3.2,1.83) circle (1.5pt);
			\draw [fill=black] (3.22,1.81) circle (1.5pt);
			\draw [fill=black] (4.82,1.81) circle (1.5pt);
			\draw [fill=black] (-4,2) circle (1.5pt);
			\draw [fill=black] (4,2) circle (1.5pt);
			
			\draw (0,0.43) node {$v$};
			\draw (-7,-0.43) node {$u_0$};
			\draw (-1,-0.43) node {$u_{s_1}=a$};
			\draw (1,-0.43) node {$v_0=b$};
			\draw (7,-0.43) node {$v_{s_2}$};
			\draw (-3.84,2.84) node {$k_1$};
			\draw (4.14,2.84) node {$k_2$};
			\draw (-8,0) node {$T^*\hspace{2pt}= $};
		\end{tikzpicture}
		\caption{if $T^*$ has a spectral vertex $v$}
	\end{subfigure}
	\caption{}
	\label{fig:T^*_structure}
\end{figure}

So we have shown that the extremal tree $T^*$ looks like the trees shown in \autoref{fig:T^*_structure}. In what follows, we determine the parameters $k_i, \ell_i$ and $r_i$ for all $(n,d)$. Since $H_1$ and $H_2$ are not isomorphic to $K_1$, we can assume that $\ell_1\ge 1$ and $r_2\ge 1$. We first resolve the cases $d=3$ and $d=4$.

\begin{claim}\label{claim:d=3}
	If $d=3$ then $\ell_1=1=r_2$, $r_1=0=\ell_2$ and $|k_1-k_2|\le 1$. 
\end{claim}

\begin{proof}
	Since $d=3$ and $\lambda_2^*>0$, $T^*$ does not have a spectral vertex. So we have $T^*=(C(1,0,k_1),u_1)\circ (C(0,1,k_2),v_0)$. We only need to show that $|k_1-k_2|\le 1$. Suppose to the contrary $k_1\ge k_2+2$. Let $T'=(C(1,0,k_1-1),u_1)\circ (C(0,1,k_2+1),v_0)$. By \autoref{lemma:characteristic_G1G2},
	\begin{align*}
		\Phi(T^*,x)&=x^{k_1+k_2}\big[ x^4-(k_1+k_2+3)x^2+(k_1+1)(k_2+1)\big];\\
		\Phi(T',x)&=x^{k_1+k_2} \big[ x^4-(k_1+k_2+3)x^2+ k_1(k_2+2)\big].
	\end{align*}
	Since $\lambda_2^*>0$, we have
	$$
	\Phi(T',\lambda_2^*) = \Phi(T',\lambda_2^*)-\Phi(T^*,\lambda_2^*) =
	(\lambda_2^*)^{k_1+k_2}\big[ k_1(k_2+2)-(k_1+1)(k_2+1) \big] > 0.
	$$
	This means $\lambda_2(T')>\lambda_2^*$, a contradiction. This completes the proof.
\end{proof}

\begin{claim}\label{claim:d=4}
	If $d=4$ then $\ell_1=1=r_2$ and $|k_1-k_2|\le 1$. Moreover, if $k_1=k_2$, then $r_1=0=\ell_2$. If $k_1< k_2$ then $r_1=1$ and $\ell_2=0$.
\end{claim}

\begin{proof}
	We first show that $\ell_1=1=r_2$. Suppose to the contrary $\ell_1\ge 2$. Since $r_2\ge 1$, $\ell_1=2$ and $T^*=(C(2,0,k_1),u_2)\circ (C(0,1,k_2),v_0)$. Let $T'=(C(1,1,k_1),u_2) \circ(C(0,1,k_2),v_0)$. By \autoref{lemma:characteristic_G1G2},
	\begin{align*}
		\Phi(T^*,x)&=x^{k_1+k_2-1}\big[x^6-(k_1+k_2+4)x^4+(k_1k_2+2k_1+2k_2+3)x^2-k_1(k_2+1)\big];\\
		\Phi(T',x)&=x^{k_1+k_2-1}\big[x^6-(k_1+k_2+4)x^4+(k_1k_2+2k_1+2k_2+3)x^2\big].
	\end{align*}
	Thus,
	$$
	\Phi(T',\lambda_2^*) = \Phi(T',\lambda_2^*)-\Phi(T^*,\lambda_2^*) =
	(\lambda_2^*)^{k_1+k_2-1}(k_1)(k_2+1) > 0,
	$$
	a contradiction. We conclude that $\ell_1=1$. Similarly, we can argue that $r_2=1$. Hence, 
	\begin{equation}\label{eq:claim:d=4}
		T^*\cong (K_{1,k_1+1},a')\circ K_1 \circ (K_{1,k_2+1},b')
	\end{equation}
	where $(K_{1,k_1+1},a')\subseteq (H_1,a)$ and $(K_{1,k_1+1},b')\subseteq (H_2,b)$ are rooted stars with $a'$ and $b'$ as their central vertices, respectively. 
	
	Suppose to the contrary that $k_1\ge k_2+2$. Let $T'=(K_{1,k_1},a')\circ K_1 \circ (K_{1,k_2+2},b')$. By \autoref{lemma:characteristic_G1G2},
	\begin{align*}
		\Phi(T^*,x)&=x^{k_1+k_2+1}\big[x^4-(k_1+k_2+4)x^2+k_1k_2+2k_1+2k_2+3  \big];\\
		\Phi(T',x)&=x^{k_1+k_2+1}\big[x^4-(k_1+k_2+4)x^2+k_1k_2+3k_1+k_2+2 \big].
	\end{align*}
	Thus,
	$$
	\Phi(T',\lambda_2^*) = \Phi(T',\lambda_2^*)-\Phi(T^*,\lambda_2^*) =
	(\lambda_2^*)^{k_1+k_2+1}\big[k_1-k_2-1\big] > 0.
	$$
	This is a contradiction. It follows that $|k_1-k_2|\le 1$. 
	
	If $k_1=k_2$ then by \eqref{eq:claim:d=4} and Interlacing Theorem, $\lambda_2^*=\sqrt{k_1+2}$. So, by \autoref{thm:spectral_center}, $T^*$ has a spectral vertex and $T^*=(C(1,0,k_1),u_1)\circ K_1 \circ (C(0,1,k_2),v_0)$. If $k_2=k_1+1$, then by \autoref{thm:spectral_center}, $T^*$ cannot have a spectral vertex, so it has a spectral edge. If $\ell_2\ge 1$ then 
	\[ \lambda_1(H_2-b)=\sqrt{k_2+1}>\sqrt{k_1+1}= \lambda_1(H_1).\] This contradicts \autoref{thm:spectral_center}. Hence $T^*=(C(1,1,k_1),u_2) \circ (C(0,1,k_2),v_0)$.
\end{proof}

From now on, we assume that $d\ge 5$. We deal with the case $\lambda_2^*>2$ first and show that $\ell_i$ and $r_i$ can differ by at most 1 for each $i\in \{1,2\}$.

\begin{claim}\label{claim:l1_r1_l2_r2} Suppose $\lambda_2^*> 2$ and $d\ge 5$. 
	\begin{enumerate}[$(i)$]
		\item If $T^*=(H_1,a)\circ K_1 \circ (H_2,b)$ then $r_1-1\le \ell_1\le r_1+1$ and $\ell_2-1\le r_2\le \ell_2+1$.
		\item If $T^*=(H_1,a) \circ (H_2,b)$ then $r_1-1\le \ell_1\le r_1$ and $\ell_2-1\le r_2\le \ell_2$.
	\end{enumerate}
\end{claim}

\begin{proof}
	Since $\lambda_2^*>2$, we have $k_1\ge 1$ and $k_2\ge 1$. We shall prove the claims for $\ell_1$. The proofs for $r_2$ are similar. 
	
	First, assume $T^*=(H_1,a)\circ K_1 \circ (H_2,b)$. Suppose to the contrary that $\ell_1\notin \{r_1-1,r_1,r_1+1\}$. Let $(H_1',a')=(C(\lfloor\frac{\ell_1+r_1}{2}\rfloor,\lceil\frac{\ell_1+r_1}{2}\rceil,k_1),u_{s_1})$ be a caterpillar on the path $Q_1$. Obtain $T'$ from $T^*$ by replacing $(H_1,a)$ with $(H_1',a')$. By \autoref{thm:lambda_one_max_trees_diameter} and  \autoref{thm:spectral_center}, $\lambda_1(H_1')>\lambda_1(H_1)=\lambda_1(H_2)=\lambda_2^*$. Since $H_1'$ and $H_2$ are strongly disjoint in $T'$, by \autoref{lemma:strongly_disjoint} we have $\lambda_2(T')>\lambda_2^*$, a contradiction.
	
	Now assume $T^*=(H_1,a) \circ (H_2,b)$. Suppose to the contrary that $\ell_1\notin \{r_1-1,r_1\}$. We make the following cases: 
	
	\textbf{Case 1:} $\ell_1 < r_1-1$. ~ 
	Let $(H_1',a')=(C(\ell_1+1,r_1-1,k_1),u_{s_1})$ be a caterpillar defined on the path $Q_1$. Replace $(H_1,a)$ with $(H_1',a')$ in $T^*$ to obtain $T'$. By \autoref{lemma:characteristic_unbalanced_central_caterpillar}, $\Phi(H_1,\lambda_2^*)>\Phi(H_1',\lambda_2^*)$ and $\Phi(H_1-a,\lambda_2^*)\ge \Phi(H_1'-a',\lambda_2^*)$. Since $\lambda_1(H_2-b)<\lambda_2^*<\lambda_1(H_2)$ we have $\Phi(H_2-b,\lambda_2^*)>0>\Phi(H_2,\lambda_2^*)$. Using \autoref{lemma:characteristic_G1G2},
	\begin{align*}
		& \Phi(T',\lambda_2^*)-\Phi(T^*,\lambda_2^*)\\
		= \ & \Phi(H_2,\lambda_2^*)\big[\Phi(H_1',\lambda_2^*)-\Phi(H_1,\lambda_2^*)\big]
		- \Phi(H_2-b,\lambda_2^*)\big[\Phi(H_1'-a',\lambda_2^*)-\Phi(H_1-a,\lambda_2^*)\big] > 0.
	\end{align*}
	Since $\lambda_1(T')>\lambda_1(H_2)> \lambda_2^*$, we conclude by above that $\lambda_2(T')>\lambda_2^*$. This is a contradiction.
	
	\textbf{Case 2:} $\ell_1> r_1\ge 1$. ~
	Let $(H_1',a')=(C(r_1,\ell_1,k_1),u_{s_1})$ be a caterpillar defined on path $Q_1$. Replace $(H_1,a)$ with $(H_1',a')$ in $T^*$ to obtain $T'$. Since $H_1'\cong H_1$, we have $\Phi(H_1,\lambda_2^*)=\Phi(H_1',\lambda_2^*)$ and by \autoref{lemma:characteristic_unbalanced_central_caterpillar}, $\Phi(H_1-a,\lambda_2^*)> \Phi(H_1'-a',\lambda_2^*)$. Using
	\autoref{lemma:characteristic_G1G2},
	\[\Phi(T',\lambda_2^*)-\Phi(T^*,\lambda_2^*)= - \Phi(H_2-b,\lambda_2^*)\big[\Phi(H_1'-a',\lambda_2^*)-\Phi(H_1-a,\lambda_2^*)\big]> 0,\]
	which leads to a contradiction as in Case 1.
	
	\textbf{Case 3:} $r_1=0$ and $\ell_1\ge 2$. ~
	Let $(H_1',a')=(C(\lfloor\frac{\ell_1}{2}\rfloor,\lceil\frac{\ell_1}{2}\rceil,k_1),u_{s_1})$ be a caterpillar defined on path $Q_1$. Replace $(H_1,a)$ with $(H_1',a')$ in $T^*$ to obtain $T'$. By \autoref{lemma:characteristic_unbalanced_central_caterpillar}, $\Phi(H_1,\lambda_2^*)>\Phi(H_1',\lambda_2^*)$. Using \autoref{lemma:characteristic_central_caterpillar},
	\begin{align*}
		\Phi(H_1-a,x)-\Phi(H_1'-a',x)& = x^{k_1}\Phi(P_{\ell_1},x)-x^{k_1-1}\big(x\Phi(P_{\ell_1},x)-k_1\Phi(P_{\lfloor\frac{\ell_1}{2}\rfloor},x)\Phi(P_{\lceil\frac{\ell_1}{2}\rceil-1},x) \big)\\
		& = x^{k_1-1}k_1\Phi(P_{\lfloor\frac{\ell_1}{2}\rfloor},x)\Phi(P_{\lceil\frac{\ell_1}{2}\rceil-1},x).
	\end{align*}
	So, $\Phi(H_1-a,\lambda_2^*)> \Phi(H_1'-a',\lambda_2^*)$. Thus, as in Case 1, we get $\Phi(T',\lambda_2^*)>0$, a contradiction. 
	
	\textbf{Case 4:} $r_1=0$ and  $\ell_1=1$. ~
	If $r_2\notin \{\ell_2-1,\ell_2\}$ then either $r_2<\ell_2-1$ or $r_2>\ell_2\ge 1$ or $\ell_2=0$ and $r_2\ge 3$ (since $d\ge 5$). These cases are similar to Cases 1,2, and 3 above and thus lead to a contradiction. So we can safely assume that $r_2\in \{\ell_2-1,\ell_2\}$. Note here that $\ell_2+r_2=d-2$ and $\ell_2\ge 2$. Also, 
	\[ \sqrt{k_1+1}=\lambda_1(H_1)>\lambda_2^*>\lambda_1(H_2-b)\ge \sqrt{k_2+2}\] which implies
	$k_1\ge k_2+2$. We make the following subcases:
	
	\textbf{Subcase 4.1:} $d\ge 7$. ~
	It means $\ell_2\ge 3$. Let $(H_1',a')=(K_{1,k_1+1},a')$ where $a'$ is a leaf in $K_{1,k_1+1}$. Let $(H_2',b')=(C(\ell_2-2,r_2,k_2+1),v_0)$ defined on path $v_0\ldots v_{\ell_2+r_2-2}$. Let $T'=(H_1',a')\circ K_1 \circ (H_2',b')$. Clearly, $\lambda_1(H_1)=\sqrt{k_1+1}=\lambda_1(H_1')$. 
	Using Kelmans Operation first and then applying Contraction Lemma, one can argue that $\lambda_1(H_2) < \lambda_1(H_2')$. Since, $H_1'$ and $H_2'$ are strongly disjoint in $T'$, by \autoref{lemma:strongly_disjoint}, $\lambda_2(T')>\lambda_2^*$, a contradiction.
	
	\textbf{Subase 4.2:} $d=6$. ~
	Clearly, $\ell_2=2=r_2$. Since $\lambda_1(H_2)>\lambda_2^*>2$, it follows that $k_2\ge 2$. Let $(H_1',a')=(K_{1,k_1+1},a')$ and $(H_2',b')=(K_{1,k_2+3},b')$ where $a'$ and $b'$ are leaves in $K_{1,k_1+1}$ and $K_{1,k_2+3}$ respectively. Take $T'=(H_1',a')\circ K_1 \circ (H_2',b')$. Clearly, $\lambda_1(H_1)=\sqrt{k_1+1}=\lambda_1(H_1')$. By \autoref{prop:lambda1_central_caterpillar},
	\[\lambda_1(H_2)=\lambda_1(C(2,2,k_2)) < \sqrt{\sqrt{k_2^2+4}+2} <\sqrt{k_2+3} =\lambda_1(H_2').\] 
	Since, $H_1'$ and $H_2'$ are strongly disjoint in $T'$, by \autoref{lemma:strongly_disjoint}, $\lambda_2(T')>\lambda_2^*$. This is a contradiction.
	
	\textbf{Subcase 4.3:} $d=5$. ~
	Clearly, $\ell_2=2$ and $r_2=1$ and so $T^*\cong (K_{1,k_1+2},a')\circ (K_{1,k_2+2},b')$ where $a'$ and $b'$ are leaves in $K_{1,k_1+2}$ and $K_{1,k_2+2}$ respectively. Let $T'= (K_{1,k_1+1}, a') \circ (K_{1,k_2+3}, b')$ where $a'$ and $b'$ are leaves in $K_{1,k_1+1}$ and $K_{1,k_2+3}$ respectively. Using \autoref{lemma:characteristic_G1G2},
	\begin{align*}
		\Phi(T^*,x)&=x^{k_1+k_2}\big[x^6 - (k_1+k_2+5)x^4 + (k_1k_2+3k_1+3k_2+6)x^2 - (k_1k_2+k_1+k_2+1)\big];\\
		\Phi(T',x)&=x^{k_1+k_2}\big[x^6 - (k_1+k_2+5)x^4 + (k_1k_2+4k_1+2k_2+5)x^2 - (k_1k_2+2k_1)\big].
	\end{align*}
	Since $k_1\ge k_2+2$,
	\[\Phi(T',\lambda_2^*)-\Phi(T^*,\lambda_2^*) =
	(\lambda_2^*)^{k_1+k_2}((\lambda_2^*)^2-1)(k_1-k_2-1) > 0,\]
	a contradiction. This completes the proof.
\end{proof}

We next argue that $k_1$ and $k_2$ differ at most by 1 in $T^*$.

\begin{claim}\label{claim:k1k2} Suppose $d\ge 5$ and $\lambda_2^*>2$. Then $|k_1-k_2|\le 1$. 
\end{claim}

\begin{proof}
	Without loss of generality, we can assume $k_1\ge k_2.$ Since, $\lambda_2^*>2$, $k_1\ge k_2\ge 1$. If $T^*$ has a spectral vertex then $\lambda_1(H_1)=\lambda_1(H_2)$. By \autoref{prop:central_caterpillar_isomorphism}, $|k_1-k_2|\le 1$ and we are done. So, suppose $T^*$ has a spectral edge, and the claim is not true, i.e. $k_1\ge k_2+2.$ If $K_{1,k_2+4}$ is a subgraph of $H_1-a$ then by \autoref{prop:lambda1_central_caterpillar},
	\[ \lambda_1(H_2)< \sqrt{\sqrt{k_2^2+4}+2} \le \sqrt{k_2+4} \le \lambda_1(H_1-a).\] This is a contradiction by \autoref{thm:spectral_center}. Thus, $K_{1,k_2+4}$ is not a subgraph of $H_1-a$. Since $\ell_1\ge 1$, it follows that $k_1=k_2+2$ and $r_1=1$ using \autoref{claim:l1_r1_l2_r2}. We make the following cases:
	
	\textbf{Case 1:} $k_2\ge 2$. ~
	Since $K_{1,k_2+3}$ is a subgraph of $H_1-a$, by \autoref{prop:lambda1_central_caterpillar},
	\[ \lambda_1(H_2)< \sqrt{\sqrt{k_2^2+4}+2} \le \sqrt{k_2+3} \le \lambda_1(H_1-a),\]
	where the second inequality holds whenever $k_2\ge 2$, a contradiction.
	
	\textbf{Case 2:} $k_2=1$. ~
	Then $k_1=3$. Let $(C(1,2,2),u_3)$ and $(C(2,1,2),v_0)$ be caterpillars defined on paths $u_0u_1u_2u_3$ and $v_0v_1v_2v_3$ respectively. For $d\ge 8$, let $T'\in \mathcal{T}(n,d)$ be a tree which contains $(C(1,2,2),u_3)\circ K_1\circ (C(2,1,2),v_0)$ as a subgraph. By \autoref{prop:lambda1_central_caterpillar} we have
	\[ \lambda_2^*\le \lambda_1(H_2)< \sqrt{\sqrt{5}+2} < \lambda_1(C(1,2,2))\le \lambda_2(T').\] This means that $T^*$ is not optimal. For $d\le 7$, computer verification shows that $T^*$ is not optimal.  
	
\end{proof}

In the following two claims, we fully determine the structure of $T^*$ when $d\ge 5$ and $\lambda_2^*>2$.
\begin{claim}\label{claim:lambda2_mor2_even}
	Let $d\ge 5$ and $\lambda_2^*>2$. Suppose $n-d-1$ is even. Then $k_1=k_2=\frac{n-d-1}{2}$. Moreover, 
	\begin{enumerate}[$(i)$]
		\item if $d$ is even then $T^*=(H_1,a)\circ K_1 \circ (H_2,b)$ and $\ell_i,r_i\in \{\lfloor \frac{d-2}{4}\rfloor,\lceil\frac{d-2}{4}\rceil\}$ for $i=1,2$.
		\item if $d$ is odd then $T^*=(H_1,a) \circ (H_2,b)$, $\ell_1=r_2=\lfloor\frac{d-1}{4}\rfloor$ and $r_1=\ell_2=\lceil \frac{d-1}{4}\rceil$. 
	\end{enumerate}
\end{claim}

\begin{proof}
	By \autoref{claim:k1k2}, $k_1=k_2=\frac{n-d-1}{2}$. We make two cases:
	
	\textbf{Case 1:} $T^*$ has a spectral vertex. ~
	If $r_1 > \ell_2+1$ then by \autoref{claim:l1_r1_l2_r2}, $\ell_1\ge r_1-1\ge \ell_2+1\ge r_2$ and so $H_2$ is a subgraph of $H_1-a$. By \autoref{thm:spectral_center}, $\lambda_2^*> \lambda_1(H_1-a)\ge \lambda_1(H_2)=\lambda_2^*$, a contradiction. So, $r_1\le \ell_2+1$. By symmetry, $\ell_2 \le r_1+1$. Hence, $|r_1-\ell_2|\le 1$. Similarly, it can be shown that $|\ell_1-r_2|\le 1$. Since $\lambda_1(H_1)=\lambda_1(H_2)$, we have $\ell_1+r_1=\ell_2+r_2$. We conclude, using \autoref{claim:l1_r1_l2_r2}, that all pairs of $\ell_i, r_i$'s differ by at most 1. Since $\ell_1+r_1+\ell_2+r_2+2=d$, we conclude that $d$ is even and $\ell_i,r_i\in \{\lfloor \frac{d-2}{4}\rfloor,\lceil\frac{d-2}{4}\rceil\}$ for $i=1,2$. 
	
	\textbf{Case 2:} $T^*$ has a spectral edge. ~
	If $r_1 > \ell_2$ then by \autoref{claim:l1_r1_l2_r2}, $\ell_1\ge r_1-1\ge \ell_2\ge r_2$ and so $H_2$ is a subgraph of $H_1-a$. By \autoref{thm:spectral_center}, $\lambda_2^*> \lambda_1(H_1-a)\ge \lambda_1(H_2)\ge \lambda_2^*$, a contradiction. So, $r_1\le \ell_2$. By symmetry $r_1\ge \ell_2$. Hence, $r_1=\ell_2$. By \autoref{claim:l1_r1_l2_r2}, $\ell_1,r_2\in \{r_1-1,r_1\}$. If $\ell_1=r_1$ and $r_2=r_1-1$ then $H_1-a\cong H_2$. Similarly, if $\ell_1=r_1-1$ and $r_2=r_1$ then $H_1\cong H_2-b$. In both cases, we have a contradiction by \autoref{thm:spectral_center}. Thus, $\ell_1=r_2$. Since, $\ell_1+r_1+\ell_2+r_2+1=d$, we have that $d$ is odd. Furthermore, by \autoref{claim:l1_r1_l2_r2} and the above, $\ell_1=r_2=\lfloor\frac{d-1}{4}\rfloor$ and $r_1=\ell_2=\lceil \frac{d-1}{4}\rceil$. 
\end{proof}

\begin{claim}\label{claim:lambda2_more2_odd} Let $d\ge 5$ and $\lambda_2^*>2$. Suppose $n-d-1$ is odd and $k_1\le k_2$.
	\begin{enumerate}[$(i)$]
		\item If $d=5$ then $\ell_1=1=r_1$, $\ell_2=0$ and $r_2=1$. 
		\item If $d\ge 6$ and $n-d-1\ge 5$ then $\ell_1=\lfloor \frac{d-3}{2}\rfloor$, $r_1=\lceil \frac{d-3}{2}\rceil$, $\ell_2=1$ and $r_2=1$.
		\item If $d\ge 6$ and $n-d-1=3$ then $\ell_1=\lfloor \frac{d-4}{2}\rfloor$,  $r_1=\lceil \frac{d-4}{2}\rceil$, $\ell_2=2$ and $r_2=1$.
	\end{enumerate}
\end{claim}

\begin{proof}
	By \autoref{claim:k1k2}, without loss of generality, we can assume that $k_1=\frac{n-d-2}{2}$ and $k_2=\frac{n-d}{2}$. 
	
	Consider the case $d=5$. Then by the Interlacing Theorem $\lambda_2^*=\sqrt{k_1+2}=\sqrt{k_2+1}$. By \autoref{cor:specral_center}, $T^*$ has a spectral vertex. Hence, $\ell_1=1=r_1$ and $\ell_2+1=1=r_2$. 
	
	So assume $d\ge 6$. Suppose $T^*$ has a spectral vertex. Then $\lambda_1(H_1)=\lambda_1(H_2)$. If $r_1\ge 1$ and $\ell_2\ge 1$ then by \autoref{prop:central_caterpillar_isomorphism}, we have $k_1=k_2$, a contradiction. If $r_1=0$ then by \autoref{claim:l1_r1_l2_r2}, $\ell_1=1$ and so $\lambda_1(H_1)=\sqrt{k_1+1}<\sqrt{k_2+1}\le \lambda_1(H_2)$, a contradiction. Hence, the only possibility is $r_2\ge 1$ and $\ell_2=0$. By \autoref{claim:l1_r1_l2_r2}, $r_2=1$. Then $\sqrt{k_2+1}=\lambda_1(H_2)=\lambda_1(H_1)\ge \sqrt{k_1+2}=\sqrt{k_2+1}$, which implies all inequalities are equalities. Hence, $\ell_1=1=r_1$. It follows that $d=5$, which contradicts our assumption that $d\ge 6$. Hence, $T^*=(H_1,a) \circ (H_2,b)$. 
	
	First, suppose $n-d-1\ge 5$. Then $k_1\ge 2$. If $\ell_2\ge 2$ then, using \autoref{prop:lambda1_central_caterpillar},
	\[ \lambda_1(H_1)<\sqrt{\sqrt{k_1^2+4}+2} <\sqrt{k_1+3}=\sqrt{k_2+2}\le \lambda_1(H_2-b),\]
	a contradiction. Hence, $\ell_2=1$, and by \autoref{claim:l1_r1_l2_r2}, $r_2=1$. Since $\ell_1+r_1+\ell_2+r_2+1=d$, we have $\ell_1+r_1= d-3$. Using \autoref{claim:l1_r1_l2_r2}, $\ell_1=\lfloor \frac{d-3}{2}\rfloor$ and $r_1=\lceil \frac{d-3}{2}\rceil$. 
	
	Suppose $n-d-1=3$. Then $k_1=1$. If $r_2\ge 2$ then by \autoref{claim:l1_r1_l2_r2}, $\ell_2\ge 2$. Using \autoref{prop:lambda1_central_caterpillar}, 
	\[\lambda_1(H_1)<\sqrt{\sqrt{5}+2}<\lambda_1(C(1,2,2)) \le \lambda_1(H_2-b),\]
	a contradiction. Thus, $r_2=1$. By \autoref{claim:l1_r1_l2_r2}, $\ell_2\in \{1,2\}$. If $\ell_2=1$ then $H_2\cong C(1,1,2)$ and so 
	$\lambda_2^* < \lambda_1(H_2) = 2$, a contradiction. So, $\ell_2=2$. In this case, $\ell_1+r_1= d-4$. Using \autoref{claim:l1_r1_l2_r2}, we see that $\ell_1=\lfloor \frac{d-4}{2}\rfloor$ and $r_1=\lceil \frac{d-4}{2}\rceil$. 
\end{proof}

Finally, it remains to deal with the case when $d\ge 5$ and $\lambda_2^*\le 2$. This case is resolved mainly using computer verification.

\begin{claim}\label{claim:lambda2_less2} Let $d\ge 5$ and $k_1\le k_2$. Then $\lambda_2^*\le 2$ if and only if one of the following happens:
	\begin{enumerate}[$(i)$]
		\item $n-d-1=1$. In this case, $T^*\cong C(d-r_2,r_2,1)$ where $r_2=1$ if $d\le 10$, $r_2\in \{1,2\}$ if $d=11$, $r_2=2$ if $12\le d\le 22$, $r_2\in \{2,3\}$ if $d=23$ and $r_2=3$ if $d\ge 24$.
		\item $n-d-1=2$ and $d\le 14$, or $n-d-1=4$ and $d\le 6$. Then $k_1=k_2=\frac{n-d-1}{2}$. Moreover, if $d$ is even then $\ell_1,r_2\in \{\lfloor \frac{d-2}{4}\rfloor,\lceil\frac{d-2}{4}\rceil\}$, and if $d$ is odd then $\ell_1=r_2=\lfloor\frac{d-1}{4}\rfloor$. 
		\item $n-d-1=3$ and $d\le 10$, or  $n-d-1=5$ and $d=5$. Then $k_1+1=k_2=\frac{n-d}{2}$, $\ell_1=\lfloor \frac{d-3}{2}\rfloor$ and $r_2=1$.
	\end{enumerate}
\end{claim}

\begin{proof}
	One can check that $\lambda_2^*>2$ if $T^*\in \mathcal{T}(12,5)$. So if $\lambda_2^*\le 2$ then $n-d-1\le 5$. The cases $(ii)$ and $(iii)$ can now be dealt with using a computer (which we did). It remains to prove claim $(i)$. 
	
	Note that if $n-d-1=1$ then 
	\[\lambda_2^* \le \min\{\lambda_1(H_1),\lambda_1(H_2)\}<\lambda_1(P_d)<2\] for all $d\ge 5$. 
	Clearly, in this case,  $T^*\cong C(d-r_2,r_2,1)$, $H_1\cong P_{d-\ell_2-r_2}$ and $H_2=C(\ell_2,r_2,1)$.
	
	Assume $d\ge 40$ and suppose, for a contradiction, that $r_2\neq 3$. 
	First, if $\ell_2+r_2\ge 7$, then we consider
	$T'=C(d-3,3,1)$. Note that $P_{d-7}$ and $C(3,3,1)$ are strongly disjoint subgraphs of $T'$. Since $H_1\subseteq P_{d-7}$, we have 
	$\lambda_1(P_{d-7})\ge \lambda_1(H_1)\ge \lambda_2^*$. Moreover,
	$\lambda_1(C(3,3,1)) = 2 > \lambda_2^*$.
	By \autoref{lemma:strongly_disjoint}, $\lambda_2(T')>\lambda_2^*$, a contradiction. Second, if $2 \le \ell_2+r_2\le 6$, we proceed as follows. Since, $r_2\neq 3$ by assumption, we have that $H_2$ is a subgraph of $C(5,1,1)$ or $C(4,2,1)$. Note that 
	\[
	\lambda_1(H_2)\le \max\{\lambda_1(C(5,1,1)), \lambda_1(C(4,2,1))\}
	< \lambda_1(P_{30}) \le \lambda_1(P_{d-7}) \le \lambda_1(H_1-a),
	\]
	a contradiction. We conclude that $r_2=3$ when $d\ge 40$. 
	
	For $d\le 39$, the claim has been verified using a computer. Note that for $d=11,23$, one can first find a candidate for $T^*$ using a computer. Then, find their $\lambda_2$-eigenvectors to see that they have a zero entry. Hence, for $d=11,23$, one can argue that $T^*$ is not unique.
\end{proof}

Compiling the results from Claims \ref{claim:d=3}, \ref{claim:d=4}, \ref{claim:lambda2_mor2_even}, \ref{claim:lambda2_more2_odd} and \ref{claim:lambda2_less2} gives us \autoref{thm:lambda_two_max_trees_diameter}. And \autoref{cor:lambda_two_max_trees_diameter_bound} follows immediately using the fact that $\lambda_2^* \le \min \{\lambda_1(H_1), \lambda_1(H_2)\}$, and \autoref{prop:lambda1_central_caterpillar}.


\section{$\lambda_2$-minimization in $\mathcal{T}(n,d)$}\label{section:lambda_2_min}

In this section, we prove \autoref{thm:lambda_two_min_trees_diameter}. Assume $n\ge 4$ and $3\leq d\le n-2$. Let $\lambda_2^\#=\min\{\lambda_2(T): T\in \mathcal{T}(n,d)\}$, and let $T^\#\in \mathcal{T}(n,d)$ be a tree such that $\lambda_2(T^\#)=\lambda_2^\#.$ Since, $T^\#$ is not a star, $\lambda_2^\#>0$. Let $Q=v_0 \ldots v_d$ be a longest path in $T^\#$.

\begin{claim}\label{claim:bounds_min_lambda2}
	We have $\lambda_1(P_{\lfloor\frac{d}{2}\rfloor})\le \lambda_2^\# \le \lambda_1(P_{\lceil\frac{d}{2}\rceil})$. 
\end{claim}

\begin{proof}
	Since $Q\cong P_{d+1}$ is an induced subgraph of $T^\#$, the lower bound for $\lambda_2^\#$ follows by the Interlacing Theorem. The upper bound is immediate from \autoref{prop:caterpillar_lambda_2}.   
\end{proof}

In the next two claims, we argue that $Q$ has at most one vertex with degree 3 or more in $T^\#$.

\begin{claim}
	If $i\notin \{\lfloor\frac{d}{2}\rfloor, \lfloor\frac{d}{2}\rfloor+1\}$ then $\deg(v_i)\le 2$ in $T^\#$. 
\end{claim}

\begin{proof}
	Suppose to the contrary that there is an $1\le i\le \lfloor\frac{d}{2}\rfloor-1$ such that $\deg(v_i)\ge 3$. Let $C$ be the component of $T^\#-v_{\lfloor\frac{d}{2}\rfloor}$ which contains $v_i$. Then $C$ and $P_{\lceil\frac{d}{2}\rceil}$ are strongly disjoint subgraphs in $T^\#$. Moreover, $\lambda_1(C)>\lambda_1(P_{\lceil\frac{d}{2}\rceil})$ by Kelmans Operation. Thus, by \autoref{lemma:strongly_disjoint}, $\lambda_2^\#>\lambda_1(P_{\lceil\frac{d}{2}\rceil})$. This contradicts \autoref{claim:bounds_min_lambda2}. Similarly, if $d\ge i\ge \lfloor\frac{d}{2}\rfloor+2$ then $\deg(v_i)\le 2$. 
\end{proof}

\begin{claim}\label{claim:deg_atleast_3}
	$Q$ has a unique vertex with degree at least 3 in $T^\#$.     
\end{claim}
\begin{proof}
	Let $a_1=v_{\lfloor\frac{d}{2}\rfloor}$ and $a_2=v_{\lfloor\frac{d}{2}\rfloor+1}$. Suppose to the contrary that both $a_1$ and $a_2$ have degree at least 3. Let $B_i$ denote the maximal subtree rooted at $a_i$ such that $V(B_i)\cap V(Q-a_i)=\emptyset$. We make the following cases:
	
	\textbf{Case 1:} $d$ is even. ~ Let $C$ be the component of $T^\#-a_1$ containing $a_2$. Then $P_{\frac{d}{2}}$ and $C$ are strongly disjoint in $T^\#$ and $\lambda_1(C)>\lambda_1(P_{\frac{d}{2}})$. By \autoref{lemma:strongly_disjoint}, $\lambda_2^\#>\lambda_1(P_{\frac{d}{2}})$, which contradicts \autoref{claim:bounds_min_lambda2}.
	
	\textbf{Case 2:} $d$ is odd. ~ Obtain $T'$ from $T^\#$ by deleting edges $a_2w$ and adding edges $a_1w$ for all $w\in N(a_2)\cap V(B_2)$. First using claim $(i)$ of \autoref{lemma:characteristic_G1G2} with edge $a_1a_2$ and then using claim $(iii)$ of \autoref{lemma:characteristic_G1G2}, we can find the following characteristic polynomials of $T^\#$ and $T'$:
	\begin{align*}
		\Phi(T^\#,x) =
		& \Bigl(\Phi(B_1-a_1,x)\Phi(P_{\frac{d+1}{2}},x)+\Phi(B_1,x)\Phi(P_{\frac{d-1}{2}},x)-x\Phi(B_1-a_1,x)\Phi(P_{\frac{d-1}{2}},x)\Bigr)\\
		& \times \Bigl( \Phi(B_2-a_2,x)\Phi(P_{\frac{d+1}{2}},x)+\Phi(B_2,x)\Phi(P_{\frac{d-1}{2}},x)-x\Phi(B_2-a_2,x)\Phi(P_{\frac{d-1}{2}},x)\Bigr)\\
		& - \Phi(B_1-a_1,x)\Phi(B_2-a_2,x)\Phi(P_{\frac{d-1}{2}},x)\Phi(P_{\frac{d-1}{2}},x);\\
		\Phi(T',x) =
		& \Bigl(\Phi(B_1-a_1,x)\Phi(B_2-a_2,x)\Phi(P_{\frac{d+1}{2}},x)+ \Bigl[\Phi(B_1-a_1,x)\Phi(B_2,x)\\
		& + \Phi(B_1,x)\Phi(B_2-a_2,x)-x\Phi(B_1-a_1,x)\Phi(B_2-a_2,x)\Bigr]\Phi(P_{\frac{d-1}{2}},x)\\
		& - x\Phi(P_{\frac{d-1}{2}},x)\Phi(B_1-a_1,x)\Phi(B_2-a_2,x)\Bigr)\Phi(P_{\frac{d+1}{2}},x)\\
		& - \Phi(B_1-a_1,x)\Phi(B_2-a_2,x)\Phi(P_{\frac{d-1}{2}},x)\Phi(P_{\frac{d-1}{2}},x).
	\end{align*}
	Thus, 
	\begin{align}\label{eq:d_odd_deg}
		& \quad \Phi(T',\lambda_2^\#) \nonumber \\
		& = \Phi(T',\lambda_2^\#)-\Phi(T^\#,\lambda_2^\#) \nonumber\\
		&=- \Phi(P_{\frac{d-1}{2}},\lambda_2^\#)^2\big[\Phi(B_1,\lambda_2^\#)-\lambda_2^\#\Phi(B_1-a_1,\lambda_2^\#)\big]\big[\Phi(B_2,\lambda_2^\#)-\lambda_2^\#\Phi(B_2-a_2,\lambda_2^\#)\big].
	\end{align}
	
	If we can show that $\Phi(T',\lambda_2^\#)<0$ then that would imply $\lambda_2(T')<\lambda_2^\#$, which would give us the desired contradiction. To this end, we first claim that $a_1a_2$ is a spectral edge in $T^\#$. 
	
	Let $Y$ be the spectral center in $T^\#$. If $Y$ is not the edge $a_1a_2$ then $T^\#-Y$ contains $P_{\frac{d+3}{2}}$ as a subgraph. Thus, $\lambda_2^\#\ge \lambda_1(T^\#-Y)\ge \lambda_1(P_{\frac{d+3}{2}})$. This contradicts \autoref{claim:bounds_min_lambda2}. Thus, $a_1a_2$ is a spectral edge in $T^\#$. Thus $\lambda_2^\# > \lambda_1(P_{\frac{d-1}{2}})$ which gives $\Phi(P_{\frac{d-1}{2}},\lambda_2^\#) > 0$.
	By \autoref{thm:spectral_center}, we get $\lambda_1(B_i-a_i)< \lambda_2^\#$ for $i=1,2$. By \autoref{lemma:deleted_vertex_char} and \eqref{eq:d_odd_deg} we conclude that $\Phi(T',\lambda_2^\#) < 0$. This completes the proof.
\end{proof}

So, without loss of generality, let $a=v_{\lfloor\frac{d}{2}\rfloor}$ be the unique vertex on $Q$ with degree at least three in $T^\#$. Let $b=v_{\lfloor\frac{d}{2}\rfloor+1}$ and let $B$ denote the maximal subtree rooted at $a$ such that $V(B)\cap V(Q-a)=\emptyset$. The extremal tree $T^\#$ looks like the tree shown in \autoref{fig:T_min_structure}.

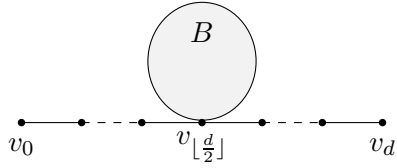
\begin{figure}[H]
	\centering
	\begin{tikzpicture}[scale=0.8]
		\draw  (-3,0)-- (-2,0);
		\draw  [dashed](-2,0)-- (-1,0);
		\draw  (-1,0)-- (0,0);
		\draw  (0,0)-- (1,0);
		\draw  [dashed](1,0)-- (2,0);
		\draw  (2,0)-- (3,0);
		\draw [rotate around={90:(0,1.02)}, fill=gray!10] (0,1.02) ellipse (0.98cm and 0.8990550594930213cm);
		\draw [fill=black] (-3,0) circle (1.5pt);
		\draw [fill=black] (-2,0) circle (1.5pt);
		\draw [fill=black] (-1,0) circle (1.5pt);
		\draw [fill=black] (0,0) circle (1.5pt);
		\draw [fill=black] (1,0) circle (1.5pt);
		\draw [fill=black] (2,0) circle (1.5pt);
		\draw [fill=black] (3,0) circle (1.5pt);
		
		\draw (0,1.5) node {$B$};
		\draw (-3,-0.4) node {$v_0$};
		\draw (0,-0.4) node {$v_{\lfloor \frac{d}{2} \rfloor}$};
		\draw (3,-0.4) node {$v_d$};
	\end{tikzpicture}
	\caption{The extremal tree $T^\#$}
	\label{fig:T_min_structure}
\end{figure}

In what follows, we try to determine the structure of the subtree $B$. A connected graph $G$ is called a \emph{Smith graph} if $\lambda_1(G) \le 2$ (see \cite[Chapter 3, Theorem 3.1.3]{Brouwer_Haemers_book}).  

\begin{claim}\label{claim:subtree_B} $\lambda_1(B-a)\le \lambda_1(P_{\lceil\frac{d}{2}\rceil}) < 2$ and so $B-a$ is a union of Smith graphs.
\end{claim}

\begin{proof}
	Note that $B-a$ and $P_{\lceil\frac{d}{2}\rceil}$ are strongly disjoint in $T^\#$. If $\lambda_1(B-a)> \lambda_1(P_{\lceil\frac{d}{2}\rceil})$ then by \autoref{lemma:strongly_disjoint}, $\lambda_2^\#>\lambda_1(P_{\lceil\frac{d}{2}\rceil})$, a contradiction. 
\end{proof}

We can say more when the parity of $d$ is known. 
\begin{claim}
	If $d$ is even then $a=v_{\frac{d}{2}}$ is a spectral vertex in $T^\#$ and $\lambda_2^\# = \lambda_1(P_{\frac{d}{2}})$. 
\end{claim}
\begin{proof}
	We have $\lambda_2^\# = \lambda_1(P_{\frac{d}{2}})= \lambda_1(T^\#-v_{\frac{d}{2}})$ using \autoref{claim:bounds_min_lambda2} and \autoref{claim:subtree_B}. By \autoref{cor:specral_center}, $v_{\frac{d}{2}}$ is a spectral vertex.
\end{proof}

So, in the even case, $B-a$ can be any union of Smith graphs, provided the diameter and the order conditions are satisfied. In the odd case for $d$, we show that $B$ has to be a subdivision of a star.  

\begin{claim}
	If $d$ is odd then $T^\#$ contains a spectral edge $ab=v_{\frac{d-1}{2}}v_{\frac{d+1}{2}}$. Moreover,
	\[\lambda_1(P_{\frac{d-1}{2}})<\lambda_2^\# < \lambda_1(P_{\frac{d+1}{2}}),\] and $B$ is a subdivision of a star with `a' as its central vertex.
\end{claim}

\begin{proof}
	Let $Y$ be the spectral center in $T^\#$. If $Y$ is not the edge $ab$ then $T-Y$ has a component containing $P_{\frac{d+1}{2}}$ and so 
	\[ \lambda_2^\#\ge \lambda_1(T^\#-Y)\ge \lambda_1(P_{\frac{d+1}{2}})>\lambda_2(C(\lfloor\frac{d}{2}\rfloor, \lceil\frac{d}{2}\rceil, n-d-1))\] by \autoref{prop:caterpillar_lambda_2}, which is a contradiction. Thus, $Y=ab$. By \autoref{thm:spectral_center}, 
	\[\lambda_1(P_{\frac{d-1}{2}})<\lambda_2^\# < \lambda_1(P_{\frac{d+1}{2}})<2.\]
	This implies 
	\begin{equation}\label{eq:char_lambda_2_evaluate}
		\Phi(P_i,\lambda_2^\#)>0 \text{ for all } i\le \frac{d-1}{2}.
	\end{equation}
	
	Now, we show that $B$ is a subdivision of a star with $a$ being the central vertex. Suppose not. Then, there exists a vertex $v\neq a$ in $B$ such that $\deg(v)\ge 3$. Let $C$ be the component of $B-a$ containing $v$ and let $k=|C|\ge 3$. We observe here that $k\le \frac{d-1}{2}$, for otherwise $\lambda_1(C)\ge \lambda_1(P_{\frac{d+1}{2}})$. Since $C$ and $P_{\frac{d+1}{2}}$ are strongly disjoint in $T^\#$, we get $\lambda_2^\# \ge  \lambda_1(P_{\frac{d+1}{2}})$ by the Interlacing Theorem, a contradiction. 
	
	By \autoref{claim:subtree_B}, $\lambda_1(C)<2$ and hence $C$ is a Smith graph (see \cite[Theorem 3.1.3]{Brouwer_Haemers_book}). It means $C$ has at most one vertex of degree at least 3 in $C$. To arrive at a contradiction, we will construct a new tree $T'$ such that $\lambda_2(T')<\lambda_2^\#$. We make the following cases:
	
	\textbf{Case 1:} $C$ is a path $P_k$. ~
	Then $v\sim a$ and $\deg(v)=3$. We can write $C=C(\ell,r,0)$ where $v$ is the central vertex of this caterpillar and $\ell\ge 1$, $r\ge 1$ and $\ell+r+1=k$. Let $H=T^\#-C$ and let $T'$ be the graph obtained from $H$ by attaching a path on $k$ vertices at $a$ by an edge (see \autoref{fig:Subdivision_B_path}). 
	
	\begin{figure}[H]
		\begin{subfigure}{0.49\textwidth}
			\centering  
			\begin{tikzpicture}[scale=0.8]
				\draw [rotate around={47.85:(1.08,1.19)}, fill=gray!10] (1.08,1.19) ellipse (1.68cm and 0.48cm);
				\draw [decorate, thick, decoration = {calligraphic brace, raise=4pt}] (-0.8,2.6) --  (-0.8,1.4) node[scale=0.8, pos=0.5, right=6pt] {$r$};
				\draw [decorate, thick, decoration = {calligraphic brace, raise=4pt}] (-1.4,0.8) --  (-3.2,0.8) node[scale=0.8, pos=0.5, below=6pt] {$\ell$};
				
				\draw  (-4,0)-- (-3,0);
				\draw  [dashed] (-3,0)-- (-1,0);
				\draw  (-1,0)-- (0,0);
				\draw  (0,0)-- (1,0);
				\draw  [dashed] (1,0)-- (3,0);
				\draw  (3,0)-- (4,0);
				\draw  (0,0)-- (-0.8,0.8);
				\draw  (-0.8,0.8)-- (-1.4,0.8);
				\draw  (-1.4,0.8)-- (-2,0.8);
				\draw  [dashed] (-2,0.8)-- (-2.6,0.8);
				\draw  (-2.6,0.8)-- (-3.2,0.8);
				\draw  (-0.8,0.8)-- (-0.8,1.4);
				\draw  [dashed] (-0.8,1.4)-- (-0.8,2);
				\draw  (-0.8,2)-- (-0.8,2.6);
				
				\draw  [fill=black] (-4,0) circle (1.5pt);
				\draw  [fill=black] (-3,0) circle (1.5pt);
				\draw  [fill=black] (-1,0) circle (1.5pt);
				\draw  [fill=black] (0,0) circle (1.5pt);
				\draw  [fill=black] (0,-0.2) node[scale=0.8] {$a$};
				\draw  [fill=black] (1,0) circle (1.5pt);
				\draw  [fill=black] (3,0) circle (1.5pt);
				\draw  [fill=black] (4,0) circle (1.5pt);
				
				\draw  [fill=black] (-0.8,0.8) circle (1.5pt);
				\draw  [fill=black] (-0.6,0.8) node[scale=0.8] {$v$};
				\draw  [fill=black] (-1.4,0.8) circle (1.5pt);
				\draw  [fill=black] (-2,0.8) circle (1.5pt);
				\draw  [fill=black] (-2.6,0.8) circle (1.5pt);
				\draw  [fill=black] (-3.2,0.8) circle (1.5pt);
				\draw  [fill=black] (-0.8,1.4) circle (1.5pt);
				\draw  [fill=black] (-0.8,2) circle (1.5pt);
				\draw  [fill=black] (-0.8,2.6) circle (1.5pt);
				\draw  [fill=black] (1.1,1.2) node[scale=0.8] {$B-C$};
			\end{tikzpicture}
			\caption{$T^\#$}
		\end{subfigure}
		\begin{subfigure}{0.49\textwidth}
			\vspace{4pt}
			\centering 
			\begin{tikzpicture}[scale=0.8]
				\draw [rotate around={47.85:(1.08,1.19)}, fill=gray!10] (1.08,1.19) ellipse (1.68cm and 0.48cm);
				\draw [decorate, thick, decoration = {calligraphic brace, raise=4pt}] (-1.4,0.8) --  (-3.2,0.8) node[scale=0.8, pos=0.5, below=6pt] {$\ell+r$};
				
				\draw  (-4,0)-- (-3,0);
				\draw  [dashed] (-3,0)-- (-1,0);
				\draw  (-1,0)-- (0,0);
				\draw  (0,0)-- (1,0);
				\draw  [dashed] (1,0)-- (3,0);
				\draw  (3,0)-- (4,0);
				\draw  (0,0)-- (-0.8,0.8);
				\draw  (-0.8,0.8)-- (-1.4,0.8);
				\draw  (-1.4,0.8)-- (-2,0.8);
				\draw  [dashed] (-2,0.8)-- (-2.6,0.8);
				\draw  (-2.6,0.8)-- (-3.2,0.8);
				
				\draw  [fill=black] (-4,0) circle (1.5pt);
				\draw  [fill=black] (-3,0) circle (1.5pt);
				\draw  [fill=black] (-1,0) circle (1.5pt);
				\draw  [fill=black] (0,0) circle (1.5pt);
				\draw  [fill=black] (0,-0.2) node[scale=0.8] {$a$};
				\draw  [fill=black] (1,0) circle (1.5pt);
				\draw  [fill=black] (3,0) circle (1.5pt);
				\draw  [fill=black] (4,0) circle (1.5pt);
				
				\draw  [fill=black] (-0.8,0.8) circle (1.5pt);
				\draw  [fill=black] (-0.6,0.8) node[scale=0.8] {$v$};
				\draw  [fill=black] (-1.4,0.8) circle (1.5pt);
				\draw  [fill=black] (-2,0.8) circle (1.5pt);
				\draw  [fill=black] (-2.6,0.8) circle (1.5pt);
				\draw  [fill=black] (-3.2,0.8) circle (1.5pt);
				\draw  [fill=black] (1.1,1.2) node[scale=0.8] {$B-C$};
			\end{tikzpicture}
			\caption{$T'$}
		\end{subfigure}
		\caption{}
		\label{fig:Subdivision_B_path}
	\end{figure}
	
	Using \autoref{lemma:characteristic_G1G2} at the edge $av$ we get
	\begin{align*}
		\Phi(T^\#,x)&=\Phi(H,x)\Phi(P_k,x)-\Phi(H-a,x)\Phi(P_\ell,x)\Phi(P_r,x);\\
		\Phi(T',x)&=\Phi(H,x)\Phi(P_k,x)-\Phi(H-a,x)\Phi(P_{k-1},x).
	\end{align*}
	Thus,
	\begin{align*}
		\Phi(T',\lambda_2^\#)-\Phi(T^\#,\lambda_2^\#)
		& = \Phi(H-a,\lambda_2^\#)\big[\Phi(P_\ell,\lambda_2^\#)\Phi(P_r,\lambda_2^\#)-\Phi(P_{k-1},\lambda_2^\#)\big]\\
		& = \Phi(H-a,\lambda_2^\#)\Phi(P_{\ell-1},\lambda_2^\#)\Phi(P_{r-1},\lambda_2^\#)<0.
	\end{align*}
	The last inequality follows using \eqref{eq:char_lambda_2_evaluate} and the fact that $\Phi(H-a,\lambda_2^\#)<0$. Since, $\lambda_3(T')<\lambda_2^\#$, we conclude from above that $\lambda_2(T')<\lambda_2^\#$, which is the required  contradiction. 
	
	\textbf{Case 2:} $C$ is not a path. ~
	Then we can write $C=C(\ell,r,1)$, where $\ell\ge 1$, $r \ge 1$, $\ell+r+2=k$ and $v$ has degree 3 in $C$. Let $v'$ be the leaf which is at distance $\ell$ from $v$ in $C$ and let $u\neq v'$ be a leaf adjacent to $v$ in $C$. The following sub-cases are possible:
	
	\textbf{Case 2.1:} $v\sim a$. ~
	Then $\deg(v)=4$. Let $H=T^\#-C$. Let $T'$ be the graph obtained from $T^\#-u$ by attaching a new leaf $u'$ at $v'$ (see \autoref{fig:Subdivision_B_notpath_1}). Using \autoref{lemma:characteristic_G1G2} at the edge $uv$ we get,
	\[\Phi(T^\#,x)=x\Phi(T^\#-u,x)-\Phi(T^\#-u-v,x)=x\Phi(T^\#-u,x)-\Phi(H,x)\Phi(P_\ell,x)\Phi(P_r,x).\]
	Using \autoref{lemma:characteristic_G1G2} first at edge $u'v'$ and then at the edge $va$, we get
	\begin{align*}
		\Phi(T',x)&=x\Phi(T'-u',x)-\Phi(T'-u'-v',x)\\
		&=x\Phi(T^\#-u,x)-\Phi(T^\#-u-v',x)\\
		&=x\Phi(T^\#-u,x)-\big[\Phi(H,x)\Phi(P_{\ell+r},x)-\Phi(H-a,x)\Phi(P_{\ell-1},x)\Phi(P_r,x)\big].
	\end{align*}
	Thus, 
	\begin{align*}
		&\quad \Phi(T',\lambda_2^\#) - \Phi(T^\#,\lambda_2^\#)\\
		&=\Phi(H,\lambda_2^\#)\big[\Phi(P_\ell,\lambda_2^\#)\Phi(P_r,\lambda_2^\#) - \Phi(P_{\ell+r},\lambda_2^\#)\big]+\Phi(H-a,\lambda_2^\#)\Phi(P_{\ell-1},\lambda_2^\#)\Phi(P_r,\lambda_2^\#)\\
		&=\Phi(H,\lambda_2^\#)\Phi(P_{\ell-1},\lambda_2^\#)\Phi(P_{r-1},\lambda_2^\#) + \Phi(H-a,\lambda_2^\#)\Phi(P_{\ell-1},\lambda_2^\#)\Phi(P_r,\lambda_2^\#)< 0.
	\end{align*}
	The last inequality holds using the facts that $\Phi(H,\lambda_2^\#)<0$, $\Phi(H-a,\lambda_2^\#)<0$ and \eqref{eq:char_lambda_2_evaluate}. This leads to a contradiction as in Case 1.
	
	\begin{figure}[H]
		\begin{subfigure}{0.49\textwidth}
			\centering  
			\begin{tikzpicture}[scale=0.8]
				\draw [rotate around={47.85:(1.08,1.19)}, fill=gray!10] (1.08,1.19) ellipse (1.68cm and 0.48cm);
				\draw [decorate, thick, decoration = {calligraphic brace, raise=4pt}] (-0.8,2.6) --  (-0.8,1.4) node[scale=0.8, pos=0.5, right=6pt] {$r$};
				\draw [decorate, thick, decoration = {calligraphic brace, raise=4pt}] (-1.4,0.8) --  (-3.2,0.8) node[scale=0.8, pos=0.5, below=6pt] {$\ell$};
				
				\draw  (-4,0)-- (-3,0);
				\draw  [dashed] (-3,0)-- (-1,0);
				\draw  (-1,0)-- (0,0);
				\draw  (0,0)-- (1,0);
				\draw  [dashed] (1,0)-- (3,0);
				\draw  (3,0)-- (4,0);
				\draw  (0,0)-- (-0.8,0.8);
				\draw  (-0.8,0.8)-- (-1.4,0.8);
				\draw  (-1.4,0.8)-- (-2,0.8);
				\draw  [dashed] (-2,0.8)-- (-2.6,0.8);
				\draw  (-2.6,0.8)-- (-3.2,0.8);
				\draw  (-0.8,0.8)-- (-0.8,1.4);
				\draw  [dashed] (-0.8,1.4)-- (-0.8,2);
				\draw  (-0.8,2)-- (-0.8,2.6);
				\draw  (-0.8,0.8)-- (-1.4,1.4);
				
				\draw  [fill=black] (-4,0) circle (1.5pt);
				\draw  [fill=black] (-3,0) circle (1.5pt);
				\draw  [fill=black] (-1,0) circle (1.5pt);
				\draw  [fill=black] (0,0) circle (1.5pt);
				\draw  [fill=black] (0,-0.2) node[scale=0.8] {$a$};
				\draw  [fill=black] (1,0) circle (1.5pt);
				\draw  [fill=black] (3,0) circle (1.5pt);
				\draw  [fill=black] (4,0) circle (1.5pt);
				
				\draw  [fill=black] (-0.8,0.8) circle (1.5pt);
				\draw  [fill=black] (-0.6,0.8) node[scale=0.8] {$v$};
				\draw  [fill=black] (-1.4,0.8) circle (1.5pt);
				\draw  [fill=black] (-2,0.8) circle (1.5pt);
				\draw  [fill=black] (-2.6,0.8) circle (1.5pt);
				\draw  [fill=black] (-3.2,0.8) circle (1.5pt);
				\draw  [fill=black] (-3.2,1.05) node[scale=0.8] {$v'$};
				\draw  [fill=black] (-0.8,1.4) circle (1.5pt);
				\draw  [fill=black] (-0.8,2) circle (1.5pt);
				\draw  [fill=black] (-0.8,2.6) circle (1.5pt);
				\draw  [fill=black] (-1.4,1.4) circle (1.5pt);
				\draw  [fill=black] (-1.4,1.6) node[scale=0.8] {$u$};
				\draw  [fill=black] (1.1,1.2) node[scale=0.8] {$B-C$};
			\end{tikzpicture}
			\caption{$T^\#$}
		\end{subfigure}
		\begin{subfigure}{0.49\textwidth}
			\centering \begin{tikzpicture}[scale=0.8]
				\draw [rotate around={47.85:(1.08,1.19)}, fill=gray!10] (1.08,1.19) ellipse (1.68cm and 0.48cm);
				\draw [decorate, thick, decoration = {calligraphic brace, raise=4pt}] (-0.8,2.6) --  (-0.8,1.4) node[scale=0.8, pos=0.5, right=6pt] {$r$};
				\draw [decorate, thick, decoration = {calligraphic brace, raise=4pt}] (-1.4,0.8) --  (-3.2,0.8) node[scale=0.8, pos=0.5, below=6pt] {$\ell$};
				
				\draw  (-4,0)-- (-3,0);
				\draw  [dashed] (-3,0)-- (-1,0);
				\draw  (-1,0)-- (0,0);
				\draw  (0,0)-- (1,0);
				\draw  [dashed] (1,0)-- (3,0);
				\draw  (3,0)-- (4,0);
				\draw  (0,0)-- (-0.8,0.8);
				\draw  (-0.8,0.8)-- (-1.4,0.8);
				\draw  (-1.4,0.8)-- (-2,0.8);
				\draw  [dashed] (-2,0.8)-- (-2.6,0.8);
				\draw  (-2.6,0.8)-- (-3.2,0.8);
				\draw  (-0.8,0.8)-- (-0.8,1.4);
				\draw  [dashed] (-0.8,1.4)-- (-0.8,2);
				\draw  (-0.8,2)-- (-0.8,2.6);
				\draw  (-3.8,0.8)-- (-3.2,0.8);
				
				\draw  [fill=black] (-4,0) circle (1.5pt);
				\draw  [fill=black] (-3,0) circle (1.5pt);
				\draw  [fill=black] (-1,0) circle (1.5pt);
				\draw  [fill=black] (0,0) circle (1.5pt);
				\draw  [fill=black] (0,-0.2) node[scale=0.8] {$a$};
				\draw  [fill=black] (1,0) circle (1.5pt);
				\draw  [fill=black] (3,0) circle (1.5pt);
				\draw  [fill=black] (4,0) circle (1.5pt);
				
				\draw  [fill=black] (-0.8,0.8) circle (1.5pt);
				\draw  [fill=black] (-0.6,0.8) node[scale=0.8] {$v$};
				\draw  [fill=black] (-1.4,0.8) circle (1.5pt);
				\draw  [fill=black] (-2,0.8) circle (1.5pt);
				\draw  [fill=black] (-2.6,0.8) circle (1.5pt);
				\draw  [fill=black] (-3.2,0.8) circle (1.5pt);
				\draw  [fill=black] (-3.2,1.05) node[scale=0.8] {$v'$};
				\draw  [fill=black] (-0.8,1.4) circle (1.5pt);
				\draw  [fill=black] (-0.8,2) circle (1.5pt);
				\draw  [fill=black] (-0.8,2.6) circle (1.5pt);
				\draw  [fill=black] (-3.8,0.8) circle (1.5pt);
				\draw  [fill=black] (-3.8,1.05) node[scale=0.8] {$u'$};
				\draw  [fill=black] (1.1,1.2) node[scale=0.8] {$B-C$};
			\end{tikzpicture}
			\caption{$T'$}
		\end{subfigure}
		\caption{}
		\label{fig:Subdivision_B_notpath_1}
	\end{figure}
	
	\textbf{Case 2.2:} $v\nsim a$. ~
	Then $\deg(v)=3$. Let $\Tilde{v}$ be the neighbour of $v$, which lies on the unique path from $v$ to $a$. We can further assume that $\Tilde{v}$ does not lie on the path joining $v$ to $v'$ (if needed by flipping the role of $\ell$ and $r$ in $C=C(\ell,r,1)$). Let $\Tilde{H}$ be the graph obtained from $T^\#-u$ by deleting the path joining $v$ to $v'$ (including $v$ and $v'$). Let $T'$ be the graph obtained from $T^\#-u$ by attaching a new leaf $u'$ at $v'$ (see \autoref{fig:Subdivision_B_notpath_2}). 
	
	\begin{figure}[H]
		\begin{subfigure}{0.49\textwidth}
			\centering  
			\begin{tikzpicture}[scale=0.8]
				\draw [rotate around={47.85:(1.08,1.19)}, fill=gray!10] (1.08,1.19) ellipse (1.68cm and 0.48cm);
				\draw [decorate, thick, decoration = {calligraphic brace, raise=4pt}] (-2.3,0.8) --  (-3.2,0.8) node[scale=0.8, pos=0.5, below=6pt] {$\ell$};
				
				\draw  (-4,0)-- (-3,0);
				\draw  [dashed] (-3,0)-- (-1,0);
				\draw  (-1,0)-- (0,0);
				\draw  (0,0)-- (1,0);
				\draw  [dashed] (1,0)-- (3,0);
				\draw  (3,0)-- (4,0);
				\draw  (0,0)-- (-0.8,0.8);
				\draw  [dashed] (-0.8,0.8)-- (-1.4,0.8);
				\draw  (-1.4,0.8)-- (-2,0.8);
				\draw  [dashed] (-2,0.8)-- (-2.6,0.8);
				\draw  (-2.6,0.8)-- (-3.2,0.8);
				\draw  (-0.8,0.8)-- (-0.8,1.4);
				\draw  [dashed] (-0.8,1.4)-- (-0.8,2);
				\draw  (-0.8,2)-- (-0.8,2.6);
				\draw  (-2,0.8)-- (-2,1.4);
				
				\draw  [fill=black] (-4,0) circle (1.5pt);
				\draw  [fill=black] (-3,0) circle (1.5pt);
				\draw  [fill=black] (-1,0) circle (1.5pt);
				\draw  [fill=black] (0,0) circle (1.5pt);
				\draw  [fill=black] (0,-0.2) node[scale=0.8] {$a$};
				\draw  [fill=black] (1,0) circle (1.5pt);
				\draw  [fill=black] (3,0) circle (1.5pt);
				\draw  [fill=black] (4,0) circle (1.5pt);
				
				\draw  [fill=black] (-0.8,0.8) circle (1.5pt);
				\draw  [fill=black] (-1.4,0.8) circle (1.5pt);
				\draw  [fill=black] (-1.4,1.05) node[scale=0.8] {$\Tilde{v}$};
				\draw  [fill=black] (-2,0.8) circle (1.5pt);
				\draw  [fill=black] (-1.85,1) node[scale=0.8] {$v$};
				\draw  [fill=black] (-2.6,0.8) circle (1.5pt);
				\draw  [fill=black] (-3.2,0.8) circle (1.5pt);
				\draw  [fill=black] (-3.2,1.05) node[scale=0.8] {$v'$};
				\draw  [fill=black] (-0.8,1.4) circle (1.5pt);
				\draw  [fill=black] (-0.8,2) circle (1.5pt);
				\draw  [fill=black] (-0.8,2.6) circle (1.5pt);
				\draw  [fill=black] (-2,1.4) circle (1.5pt);
				\draw  [fill=black] (-2,1.6) node[scale=0.8] {$u$};
				\draw  [fill=black] (1.1,1.2) node[scale=0.8] {$B-C$};
			\end{tikzpicture}
			\caption{$T^\#$}
		\end{subfigure}
		\begin{subfigure}{0.49\textwidth}
			\centering 
			\begin{tikzpicture}[scale=0.8]
				\draw [rotate around={47.85:(1.08,1.19)}, fill=gray!10] (1.08,1.19) ellipse (1.68cm and 0.48cm);
				\draw [decorate, thick, decoration = {calligraphic brace, raise=4pt}] (-2.3,0.8) --  (-3.2,0.8) node[scale=0.8, pos=0.5, below=6pt] {$\ell$};
				
				\draw  (-4,0)-- (-3,0);
				\draw  [dashed] (-3,0)-- (-1,0);
				\draw  (-1,0)-- (0,0);
				\draw  (0,0)-- (1,0);
				\draw  [dashed] (1,0)-- (3,0);
				\draw  (3,0)-- (4,0);
				\draw  (0,0)-- (-0.8,0.8);
				\draw  [dashed] (-0.8,0.8)-- (-1.4,0.8);
				\draw  (-1.4,0.8)-- (-2,0.8);
				\draw  [dashed] (-2,0.8)-- (-2.6,0.8);
				\draw  (-2.6,0.8)-- (-3.2,0.8);
				\draw  (-0.8,0.8)-- (-0.8,1.4);
				\draw  [dashed] (-0.8,1.4)-- (-0.8,2);
				\draw  (-0.8,2)-- (-0.8,2.6);
				\draw  (-3.8,0.8)-- (-3.2,0.8);
				
				\draw  [fill=black] (-4,0) circle (1.5pt);
				\draw  [fill=black] (-3,0) circle (1.5pt);
				\draw  [fill=black] (-1,0) circle (1.5pt);
				\draw  [fill=black] (0,0) circle (1.5pt);
				\draw  [fill=black] (0,-0.2) node[scale=0.8] {$a$};
				\draw  [fill=black] (1,0) circle (1.5pt);
				\draw  [fill=black] (3,0) circle (1.5pt);
				\draw  [fill=black] (4,0) circle (1.5pt);
				
				\draw  [fill=black] (-0.8,0.8) circle (1.5pt);
				\draw  [fill=black] (-1.4,0.8) circle (1.5pt);
				\draw  [fill=black] (-1.4,1.05) node[scale=0.8] {$\Tilde{v}$};
				\draw  [fill=black] (-2,0.8) circle (1.5pt);
				\draw  [fill=black] (-1.85,1) node[scale=0.8] {$v$};
				\draw  [fill=black] (-2.6,0.8) circle (1.5pt);
				\draw  [fill=black] (-3.2,0.8) circle (1.5pt);
				\draw  [fill=black] (-3.2,1.05) node[scale=0.8] {$v'$};
				\draw  [fill=black] (-0.8,1.4) circle (1.5pt);
				\draw  [fill=black] (-0.8,2) circle (1.5pt);
				\draw  [fill=black] (-0.8,2.6) circle (1.5pt);
				\draw  [fill=black] (-3.8,0.8) circle (1.5pt);
				\draw  [fill=black] (-3.8,1.05) node[scale=0.8] {$u'$};
				\draw  [fill=black] (1.1,1.2) node[scale=0.8] {$B-C$};
			\end{tikzpicture}
			\caption{$T'$}
		\end{subfigure}
		\caption{}
		\label{fig:Subdivision_B_notpath_2}
	\end{figure}
	
	Using \autoref{lemma:characteristic_G1G2} at the edge $v\Tilde{v}$ we get,
	\begin{align*}
		\Phi(T^\#,x)&=\Phi(P_{\ell+2},x)\Phi(\Tilde{H},x)-x\Phi(P_\ell,x)\Phi(\Tilde{H}-\Tilde{v},x);\\
		\Phi(T',x)&=\Phi(P_{\ell+2},x)\Phi(\Tilde{H},x)-\Phi(P_{\ell+1},x)\Phi(\Tilde{H}-\Tilde{v},x).
	\end{align*}
	Thus,
	\begin{align*}
		\Phi(T',\lambda_2^\#) - \Phi(T^\#,\lambda_2^\#)
		&=\Phi(\Tilde{H}-\Tilde{v},\lambda_2^\#)\big[x\Phi(P_\ell,\lambda_2^\#) - \Phi(P_{\ell+1},\lambda_2^\#)\big]\\
		&=\Phi(\Tilde{H}-\Tilde{v},\lambda_2^\#)\Phi(P_{\ell-2},\lambda_2^\#)<0.
	\end{align*}
	The last inequality follows from the fact that $\Phi(\Tilde{H}-\Tilde{v},\lambda_2^\#)<0$ and \eqref{eq:char_lambda_2_evaluate}. This leads to a contradiction.
\end{proof}

Summarizing the above claims gives \autoref{thm:lambda_two_min_trees_diameter}. It would have been interesting if $T^\#$ turned out to be the caterpillar $C(\lfloor\frac{d}{2}\rfloor, \lceil\frac{d}{2}\rceil, n-d-1)$ whenever $d$ is odd. Unfortunately, this is false for small $n$ and $d$. Moreover, computations suggest that the following might be true.

\begin{conjecture}\label{conj:lamda_2_min}
	If $d\ge 5$ is odd, then the following is true of\/ $T^\#$.  
	\begin{enumerate}[$(i)$]
		\item If $P_k$ is a path in $B$ starting at `a' in $T^\#$ then $k\le \frac{d-3}{2}$. In other words, $T^\#$ has a unique path of length $d$ (namely $Q$).
		\item If $P_k$ and $P_\ell$ are paths in $B$ starting at `a' in $T^\#$ then $|k-\ell|\le 1$.
	\end{enumerate}
\end{conjecture}


\section{On $\lambda_2$ of trees: \autoref{thm:lambda_two_max_trees}}\label{section:old_results}
A proof of \autoref{thm:lambda_two_max_trees} was first given by Neumaier \cite{neumaier_second_1982}, but it had a mistake which was corrected by Hofmeister \cite{hofmeister_1997}. We believe it would be useful to rewrite the proof in modern terminology.

\begin{proof}[Proof of \autoref{thm:lambda_two_max_trees}]
	For $n<10$, the claim can be verified using a computer. So assume $n\ge 10$. Let $T^*$ be a $\lambda_2$-maximizer in $\mathcal{T}(n)$. Then there exist rooted subtrees $(H_1,a)$ and $(H_2,b)$ of $T^*$ as given in \autoref{thm:spectral_center}. We make the following cases:
	
	\textbf{Case 1:} $n$ is odd. ~ We have
	\[\lambda_2(T^*)\le \min\{ \lambda_1(H_1), \lambda_1(H_2)\} \le \min \{\lambda_1(K_{1,|H_1|-1}),\lambda_1(K_{1,|H_2|-1})\}\le \sqrt{\frac{n-3}{2}},\]
	and equality holds if and only if $H_1\cong H_2\cong K_{1, \frac{n-3}{2}}$. Thus, $T^*$ is any of the trees $T(\frac{n-3}{2},0,\frac{n-3}{2})$, $T(\frac{n-3}{2},0,0,\frac{n-5}{2})$ or $T(\frac{n-5}{2},0,0,0,\frac{n-5}{2})$ and they all have $\lambda_2=\sqrt{\frac{n-3}{2}}$.
	
	\textbf{Case 2:} $n$ is even. ~
	Let $T'=T(\frac{n-4}{2},0,0,\frac{n-4}{2})$. Using \autoref{lemma:characteristic_G1G2},
	\begin{align*}
		\Phi(T',x)&=x^{n-4}\bigg(x^2-\frac{n-2}{2}\bigg)^2-x^{n-6}\bigg(x^2-\frac{n-4}{2}\bigg)^2.
	\end{align*}
	Since, $\Phi(T',\sqrt{\frac{n-3}{2}})>0$ and $\lambda_1(T')>\sqrt{\frac{n-3}{2}}$, we conclude that $\lambda_2(T')>\sqrt{\frac{n-3}{2}}$. Thus,
	\begin{equation}\label{eq:lower_lambda_2}
		\lambda_2(T^*)\ge \lambda_2(T')>\sqrt{\frac{n-3}{2}}.
	\end{equation}
	Now if $T^*$ has a spectral vertex then one of the $H_i$'s has atmost $\frac{n-2}{2}$ vertices and therefore, $\lambda_2(T^*)=\lambda_1(H_i)\le \lambda_1(K_{1, |H_i|-1})\le \sqrt{\frac{n-4}{2}}$. This contradicts \eqref{eq:lower_lambda_2}. So $T^*$ has a spectral edge and
	\begin{equation}\label{eq:upper_lambda_2}
		\lambda_2(T^*)<\min\{\lambda_1(H_1), \lambda_1(H_2)\}\le \sqrt{\frac{n-2}{2}}.
	\end{equation}
	By \eqref{eq:lower_lambda_2} and \eqref{eq:upper_lambda_2}, we first see that both $H_1$ and $H_2$ are trees on $\frac{n}{2}$ vertices. Moreover, if $H_i$ is not a star then by \autoref{prop:second_largest_lambda_one_trees},
	\[\lambda_2(T^*)<\lambda_1(H_i)\le \sqrt{\frac{1}{2}\bigg(\frac{n}{2}-1+\sqrt{\frac{n^2}{4}-3n+13}\bigg)} \le \sqrt{\frac{n-3}{2}}\] 
	for $n\ge 10 $. This contradicts \eqref{eq:lower_lambda_2}. Thus, $H_1\cong H_2\cong K_{1, \frac{n-2}{2}}$, and hence $T^*$ is one of the following trees: $T'$, $T_1:=T(\frac{n-2}{2},0,\frac{n-4}{2})$, or $T_2:=T(\frac{n-2}{2},\frac{n-2}{2})$. 
	
	We claim that $T^*=T'$. Using \autoref{lemma:characteristic_G1G2},
	\begin{align*}
		\Phi(T_1,x)&=x^{n-4}\bigg(x^2-\frac{n-2}{2}\bigg)^2-x^{n-4}\bigg(x^2-\frac{n-4}{2}\bigg);\\
		\Phi(T_2,x)&=x^{n-4}\bigg(x^4+\frac{(n-2)^2}{4}-x^2(n-1)\bigg). 
	\end{align*}
	By the Interlacing Theorem, $\lambda_2(T_1)>\sqrt{\frac{n-4}{2}}$. Hence,  
	\[
	\Phi(T',\lambda_2(T_1))-\Phi(T_1,\lambda_2(T_1))=\lambda_2(T_1)^{n-6}\bigg(\lambda_2(T_1)^2-\frac{n-4}{2}\bigg)\bigg(\frac{n-4}{2}\bigg)>0.
	\]
	Thus, $\lambda_2(T')>\lambda_2(T_1)$. Using the characteristic polynomial of $T_2$ we have $\lambda_2(T_2)=(\frac{n-1-\sqrt{2n-3}}{2})^{\frac{1}{2}}$ which is less than $\sqrt{\frac{n-3}{2}}$ for $n\ge 10$. Thus, $\lambda_2(T')>\lambda_2(T_2)$. This completes the proof.
\end{proof}


\section{Concluding remarks}\label{section:conclusion}

In this paper, we focused on $\lambda_2$-optimization over the family $\mathcal{T}(n,d)$. It would be interesting to study this problem for connected graphs with a given diameter and planar graphs. The problem of minimizing $\lambda_1$ over $\mathcal{T}(n,d)$ is also largely open. Finally, we believe that more can be said about $\lambda_2$-minimizers in $\mathcal{T}(n,d)$ when $d$ is odd; see \autoref{conj:lamda_2_min}.

\bibliographystyle{plain}
\bibliography{Lreferences}

\vspace{0.6cm}
\noindent Hitesh Kumar, Email: {\tt hitesh\_kumar@sfu.ca}\\
\textsc{Dept. of Mathematics, Simon Fraser University, Burnaby, BC \ V5A 1S6, Canada}\\[2pt]

\noindent Bojan Mohar, Email: {\tt mohar@sfu.ca}\\
Supported in part by the NSERC Discovery Grant R832714 (Canada), and in part by the ERC Synergy grant KARST (European Union, ERC, KARST, project number 101071836).
On leave from FMF, Department of Mathematics, University of Ljubljana.\\
\textsc{Dept. of Mathematics, Simon Fraser University, Burnaby, BC \ V5A 1S6, Canada}\\[2pt]

\noindent Shivaramakrishna Pragada, Email: {\tt shivaramakrishna\_pragada@sfu.ca}\\
\textsc{Dept. of Mathematics, Simon Fraser University, Burnaby, BC \ V5A 1S6, Canada}\\[2pt]

\noindent Hanmeng Zhan, Email: {\tt hzhan@wpi.edu}\\
\textsc{Computer Science Department, Worcester Polytechnic Institute, Worcester, MA 01609, United States}
\end{document}